\documentclass[11pt,a4paper,draft]{article}
\usepackage{url}
\usepackage{amssymb}
\usepackage{t1enc}
\usepackage{dlfltxbcodetips} 
\usepackage{theorem}
\usepackage[latin2]{inputenc}
\usepackage[english]{babel}
\usepackage[a4paper,left=2cm,right=2cm,top=2.5cm,bottom=2.5cm]{geometry}
\usepackage[all]{xy}
\usepackage[margin=1.5cm]{caption} 

\newtheorem{theorem}{Theorem}[section]

\newtheorem{corollary}[theorem]{Corollary}

\newtheorem{lemma}[theorem]{Lemma}

\newtheorem{proposition}[theorem]{Proposition}

{\theorembodyfont{\rmfamily}\newtheorem{remark}[theorem]{Remark}}

\def\proofof[#1]{\noindent\textbf{#1}}

\begin{document}

\newcommand{\cov}{\operatorname{cov}}
\newcommand{\ran}{\operatorname{ran}}
\newcommand{\dom}{\operatorname{dom}}
\newcommand{\rank}{\operatorname{rank}}
\newcommand{\supp}{\mbox{\rm supp}}
\newcommand{\grad}{\mbox{\rm grad}}
\newcommand{\dist}{\mbox{\rm dist}}
\newcommand{\cl}{\mbox{\rm cl}}
\newcommand{\bd}{\mbox{\rm bd}}
\newcommand{\sgn}{\mbox{\rm sgn}}
\newcommand{\excess}{\mbox{\rm excess}}
\newcommand{\Id}{\mbox{\rm Id}}
\newcommand{\id}{\mbox{\rm id}}
\newcommand{\len}{\mbox{\rm len}}
\newcommand{\inarc}{\operatorname{in}}
\newcommand{\outarc}{\operatorname{out}}
\newcommand{\inarb}{\operatorname{inarb}}
\newcommand{\outarb}{\operatorname{outarb}}
\newcommand{\iness}{\operatorname{iness}}
\newcommand{\ess}{\operatorname{ess}}
\newcommand{\lspan}{\mbox{\rm span}}
\newcommand{\str}{\operatorname{str}}
\newcommand{\weak}{\operatorname{weak}}
\newcommand{\inter}{\operatorname{int}}
\newcommand{\arctg}{\operatorname{arctg}}
\newcommand{\tv}{\operatorname{tv}}
\newcommand{\diag}{\operatorname{diag}}
\newcommand{\argmax}{\operatorname{arg max}}
\newcommand{\argmin}{\operatorname{arg min}}
\newcommand{\con}{\operatorname{con}}

\newcommand{\bs}{\backslash}
\newcommand{\mb}{\mathbf}
\newcommand{\mc}{\mathcal}
\newcommand{\ms}{\mathscr}
\newcommand{\mf}{\mathfrak}
\newcommand{\ol}{\overline}
\newcommand{\ul}{\underline}
\newcommand{\orarr}{\overrightarrow}
\newcommand{\wh}{\widehat}
\newcommand{\wt}{\widetilde}
\newcommand{\eps}{\varepsilon}
\newcommand{\vt}{\vartheta}

\newcommand{\To}{\Rightarrow}
\newcommand{\ot}{\leftarrow}
\newcommand{\oT}{\Leftarrow}
\newcommand{\otto}{\leftrightarrow}
\newcommand{\oTTo}{\Leftrightarrow}

\newcommand{\R}{\mathbb{R}}
\newcommand{\RPL}{\mathbb{R}_+}
\newcommand{\RNN}{\mathbb{R}_{\geq0}}

\newcommand{\Rn}{\mathbb{R}^n}
\newcommand{\RPLn}{\mathbb{R}^n_+}
\newcommand{\RNNn}{\mathbb{R}^n_{\geq0}}
\newcommand{\Rnn}{\mathbb{R}^{n\times n}}
\newcommand{\RPLnn}{\mathbb{R}^{n\times n}_+}
\newcommand{\RNNnn}{\mathbb{R}^{n\times n}_{\geq0}}

\newcommand{\Rc}{\mathbb{R}^c}
\newcommand{\RPLc}{\mathbb{R}^c_+}
\newcommand{\RNNc}{\mathbb{R}^c_{\geq0}}
\newcommand{\Rcc}{\mathbb{R}^{c\times c}}
\newcommand{\RPLcc}{\mathbb{R}^{c\times c}_+}
\newcommand{\RNNcc}{\mathbb{R}^{c\times c}_{\geq0}}

\newcommand{\Rm}{\mathbb{R}^m}
\newcommand{\RPLm}{\mathbb{R}^m_+}
\newcommand{\RNNm}{\mathbb{R}^m_{\geq0}}
\newcommand{\Rmm}{\mathbb{R}^{m\times m}}
\newcommand{\RPLmm}{\mathbb{R}^{m\times m}_+}
\newcommand{\RNNmm}{\mathbb{R}^{m\times m}_{\geq0}}

\newcommand{\Rp}{\mathbb{R}^p}
\newcommand{\RPLp}{\mathbb{R}^p_+}
\newcommand{\RNNp}{\mathbb{R}^p_{\geq0}}
\newcommand{\Rpp}{\mathbb{R}^{p\times p}}
\newcommand{\RPLpp}{\mathbb{R}^{p\times p}_+}
\newcommand{\RNNpp}{\mathbb{R}^{p\times p}_{\geq0}}

\newcommand{\Rq}{\mathbb{R}^q}
\newcommand{\RPLq}{\mathbb{R}^q_+}
\newcommand{\RNNq}{\mathbb{R}^q_{\geq0}}
\newcommand{\Rqq}{\mathbb{R}^{q\times q}}
\newcommand{\RPLqq}{\mathbb{R}^{q\times q}_+}
\newcommand{\RNNqq}{\mathbb{R}^{q\times q}_{\geq0}}

\newcommand{\Rcn}{\mathbb{R}^{c\times n}}
\newcommand{\RPLcn}{\mathbb{R}^{c\times n}_+}
\newcommand{\RNNcn}{\mathbb{R}^{c\times n}_{\geq0}}
\newcommand{\Rnc}{\mathbb{R}^{n\times c}}
\newcommand{\RPLnc}{\mathbb{R}^{n\times c}_+}
\newcommand{\RNNnc}{\mathbb{R}^{n\times c}_{\geq0}}

\newcommand{\Rcm}{\mathbb{R}^{c\times m}}
\newcommand{\RPLcm}{\mathbb{R}^{c\times m}_+}
\newcommand{\RNNcm}{\mathbb{R}^{c\times m}_{\geq0}}
\newcommand{\Rmc}{\mathbb{R}^{m\times c}}
\newcommand{\RPLmc}{\mathbb{R}^{m\times c}_+}
\newcommand{\RNNmc}{\mathbb{R}^{m\times c}_{\geq0}}

\newcommand{\Rnm}{\mathbb{R}^{n\times m}}
\newcommand{\RPLnm}{\mathbb{R}^{n\times m}_+}
\newcommand{\RNNnm}{\mathbb{R}^{n\times m}_{\geq0}}
\newcommand{\Rmn}{\mathbb{R}^{m\times n}}
\newcommand{\RPLmn}{\mathbb{R}^{m\times n}_+}
\newcommand{\RNNmn}{\mathbb{R}^{m\times n}_{\geq0}}

\newcommand{\C}{\mathbb{C}}
\newcommand{\F}{\mathbb{F}}

\newcommand{\Z}{\mathbb{Z}}
\newcommand{\ZPL}{\mathbb{Z}_+}
\newcommand{\ZNN}{\mathbb{Z}_{\geq0}}

\newcommand{\Zn}{\mathbb{Z}^n}
\newcommand{\ZPLn}{\mathbb{Z}^n_+}
\newcommand{\ZNNn}{\mathbb{Z}^n_{\geq0}}
\newcommand{\Znn}{\mathbb{Z}^{n\times n}}
\newcommand{\ZPLnn}{\mathbb{Z}^{n\times n}_+}
\newcommand{\ZNNnn}{\mathbb{Z}^{n\times n}_{\geq0}}

\newcommand{\Zcn}{\mathbb{Z}^{c\times n}}
\newcommand{\ZPLcn}{\mathbb{Z}^{c\times n}_+}
\newcommand{\ZNNcn}{\mathbb{Z}^{c\times n}_{\geq0}}
\newcommand{\Znc}{\mathbb{Z}^{n\times c}}
\newcommand{\ZPLnc}{\mathbb{Z}^{n\times c}_+}
\newcommand{\ZNNnc}{\mathbb{Z}^{n\times c}_{\geq0}}

\newcommand{\N}{\mathbb{N}}

\title{On the dependence of the existence of the positive steady states \\ on the rate coefficients for deficiency-one mass action systems:\\ single linkage class}
\author{Bal\'azs Boros\footnote{Institute of Mathematics, E\"otv\"os Lor\'and University, Budapest, Hungary;
address: P\'azm\'any P\'eter s\'et\'any 1/c, 1117 Budapest, Hungary; email: bboros@cs.elte.hu; phone: +3613722500 (ext 8521); fax: +3613812152}}
\maketitle

\begin{abstract}
  The Deficiency-One Theorem states that there exists a unique positive steady state in each positive stoichiometric class for weakly reversible deficiency-one mass action systems with one linkage class (regardless of the values of the rate coefficients). The non-emptiness of the set of positive steady states does not remain valid if we omit the weak reversibility. A recently published paper provided an equivalent condition to the existence of a positive steady state for deficiency-one mass action systems that are not weakly reversible, but still has only one linkage class. Based on that result, we characterise in this paper those of these mass action systems for which the non-emptiness of the set of positive steady states holds regardless of the values of the rate coefficients. Also, we provide an equivalent condition to the existence of rate coefficients such that the set of positive steady states is nonempty for the resulting mass action system.

\end{abstract}

\section{Introduction}

The foundations of Chemical Reaction Network Theory (CRNT) was developed by Feinberg, Horn, and Jackson \cite{feinberg:1979, feinberg:1973, feinberg:1987, feinberg:1995a, feinberg:horn:1977, horn:1973, horn:jackson:1972} in the 1970s. Several results are available concerning the existence and/or the uniqueness of the positive steady states of mass action systems \cite{banaji:craciun:2010, banaji:donnell:baignent:2007, boros:2010, boros:2012, craciun:feinberg:2005, craciun:feinberg:2006, craciun:feinberg:2010, craciun:helton:williams:2008, feinberg:1987, feinberg:1995a}. A recent result provides an equivalent condition to the non-emptiness of the set of positive steady states for single linkage class deficiency-one mass action systems that are not weakly reversible, see \cite[Theorem III.7]{boros:2010} and \cite[Theorems 3.11 and 3.19]{boros:2012}. For such mass action systems, the non-emptiness of the set of positive steady states \emph{may} depend on the values of the rate coefficients. In this paper, we characterise those of these chemical reaction networks, which has the property that the non-emptiness of the set of positive steady states does \emph{not} depend on the values of the rate coefficients. The main graph theoretical tool we use is the Matrix-Tree Theorem, see \cite[Theorem 3.6]{tutte:1948}. We remark that there are several recent papers in the field of CRNT that makes use of the Matrix-Tree Theorem, see e.g. \cite{craciun:dickenstein:shiu:sturmfels:2008, dickenstein:perezmillan:2011, feliu:wiuf:2012, feliu:wiuf:2013, karp:perezmillan:dasgupta:dickenstein:gunawardena:2012, thomson:gunawardena:2009}. Finally, we mention here that a potential application of our results is to check the satisfaction of one of the assumptions of the main result of \cite{shinar:feinberg:2010} (see the theorem at the bottom  of page 1390 in \cite{shinar:feinberg:2010}). A certain robustness property is proven there for deficiency-one mass action systems that are not weakly reversible, admit a positive steady state, and satisfy another (easily checkable) condition.

The rest of this paper is organised as follows. After a brief section on notations, we provide an overview of the required notations from CRNT in Section \ref{sec:CRNT}. The primary goal of this paper is to investigate the dependence of the non-emptiness of the set of positive steady states on the rate coefficients for deficiency-one mass action systems that are not weakly reversible, but still has only one linkage class. In Section \ref{sec:specialcases}, we examine this dependence under some extra assumptions on the graph of complexes (the directed graph that encodes the reactions). In Subsection \ref{subsec:chain}, we assume that the reactions form a \lq\lq chain'', while in Subsection \ref{subsec:tree-like}, we pose the condition that the graph of complexes is \lq\lq tree-like''. In Sections \ref{sec:general_exists} and \ref{sec:general_forall}, we generalise the results of Section \ref{sec:specialcases} (i.e., in these sections we omit the restrictive assumptions on the structure of the graph of complexes that were posed in Section \ref{sec:specialcases}). As it will become apparent, especially in Sections \ref{sec:specialcases}, \ref{sec:general_exists}, and \ref{sec:general_forall}, we need more involved graph theoretical arguments. Therefore, Appendices \ref{sec:app_directed_graphs}, \ref{sec:app_positive_h-transshipment}, and \ref{sec:app_mtx-tree-thm} are devoted to summarise the purely graph theoretical notions and results that are used throughout this paper.

\section{Notations}

Denote by $\R$, $\RPL$, and $\RNN$ the set of real, positive real, and nonnegative real numbers, respectively, i.e., $\RPL=\{x\in\R~|~x>0\}$ and $\RNN=\{x\in\R~|~x\geq0\}$. For $v\in\R^p$, the $i$th coordinate of $v$ is denoted by $v_i$ ($i \in \{1, \ldots, p\}$). For a matrix $A\in\R^{p\times q}$, the $j$th column and the $(i,j)$th entry of $A$ are denoted by $A_{\cdot j}$ and $A_{ij}$, respectively ($i \in \{1, \ldots, p\}, j \in \{1, \ldots, q\}$). For a matrix $A$, we denote by $A^{\top}, \ker A$, and $\ran A$ the transpose, the kernel, and the range of $A$, respectively. For a square matrix $A$, denote by $\det A$ the determinant of $A$.

For any finite set $X$, $X_0 \subseteq X$, and function $g : X \to \R$ we define $g(X_0)$ by
\begin{align}\label{eq:p(X_0)=sump(x)}
g(X_0) = \sum_{x \in X_0} g(x).
\end{align}
We make this convention in order to ease the notation in several situations.

For any set $A$, denote the cardinality of $A$ by $|A|$.

We also summarise here the graph theoretical notations that will be used throughout this paper. For more details on these, please refer to Appendix \ref{sec:app_directed_graphs}. Let $D = (V, A)$ be a directed graph for the rest of this section. For $i, j \in V$, the set of directed paths from $i$ to $j$ is denoted by $\orarr{i,j}$. For a directed path $P$ we denote by $V[P]$ the vertex set of $P$ (we use the notation $V[P]$ even if the vertex set of the directed graph in question is denoted by some other symbol than $V$) and by $\len(P)$ the length of $P$. For $i_1, i_2, i_3 \in V$, $P \in \orarr{i_1, i_2}$, and $Q \in \orarr{i_2, i_3}$, denote by $\con(P, Q)$ the concatenation of $P$ and $Q$.

For $U \subseteq V$ let us denote by $\varrho^{\inarc}_D(U)$ and $\varrho^{\outarc}_D(U)$ set of arcs that enter $U$ and leave $U$, respectively. For a function $z : A \to \R$, denote by $\excess_z$ the excess function associated to $D$ and $z$. For $i \in V$ we use the notations $\varrho^{\inarc}(i)$, $\varrho^{\outarc}(i)$, and $\excess_z(i)$ instead of $\varrho^{\inarc}(\{i\})$, $\varrho^{\outarc}(\{i\})$, and $\excess_z(\{i\})$, respectively. For $U \subseteq V$, denote by $\mc{T}_D(U)$ the set of $U$-branchings in $D$. Finally, for $i, j \in V$ and $U \subseteq V$ let us define $\mc{T}^{ij}_D(U)$ by
\begin{align}
\mc{T}^{ij}_{D}(U) = \{ \wt{A} \in \mc{T}_{D}(U) ~|~ \mbox{there exists a directed path from $i$ to $j$ in $(V,\wt{A})$} \}.
\end{align}

\section{Preliminaries from CRNT}\label{sec:CRNT}

In this section we provide a brief overview of the basic notions of CRNT used in this paper. However, since the results of this paper rely heavily on the earlier result \cite[Theorem 3.11]{boros:2012}, we assume some familiarity of the reader with these notions. Therefore, this section is rather a summary of the notions and notations used later on in this paper. The interested reader should consult e.g. \cite[Sections 2 and 3]{feinberg:1995a} for a fuller discussion on these basic notions, while a short introduction can be found e.g. in \cite[Section 2]{boros:2012}. The latter reference uses almost the same notations as the present paper.

A \emph{chemical reaction network} (or just \emph{reaction network}) is a triple
$(\mc{X},\mc{C},\mc{R})$ of three nonempty finite
sets, where
\begin{itemize}
  \item $\mc{X} = \{X_1, \ldots, X_n\}$ is the \emph{set of species},
  \item $\mc{C} = \{C_1, \ldots, C_c\}$ is the \emph{set of complexes}, and
  \item $\mc{R} \subseteq \{ (C_i,C_j) \in \mc{C} \times \mc{C} ~|~ i, j \in \{1, \ldots, c\}, i \neq j \}$ is the \emph{set of reactions}.
\end{itemize}
Roughly speaking, the complexes are linear combinations of the species with nonnegative integer coefficients. A convenient way to specify the set of complexes is to provide an $n \times c$ matrix, denoted by $B$, whose entries are nonnegative integers, the rows refer to the species, and the columns refer to the complexes. In order to ease the notation, we identify the sets $\mc{C}$ and $\{1, 2, \ldots, c\}$. Accordingly, we mostly write $i$ and $(i,j)$ instead of $C_i$ and $(C_i, C_j)$, respectively ($i, j \in \{1, 2, \ldots, c\}$).

The directed graph $(\mc{C},\mc{R})$ is called the \emph{graph of complexes}. Certain properties of this directed graph will play a central role in this paper. For further reference, let us denote $(\mc{C},\mc{R})$ by the symbol $\mc{D}$.

A \emph{mass action system} is a quadruple $(\mc{X}, \mc{C}, \mc{R}, \kappa)$, where $(\mc{X}, \mc{C}, \mc{R})$ is a chemical reaction network and $\kappa : \mc{R} \to \RPL$ is any function. For $(i,j) \in \mc{R}$, the value $\kappa_{(i,j)}$ is called the \emph{rate coefficient} of the reaction $(i,j)$. In the sequel, we write $\kappa_{ij}$ instead of $\kappa_{(i,j)}$. We consider a continuous-time continuous-state deterministic model, where the state of the system represents the \emph{concentrations} of the species and the time-evolution of the state is described by the ordinary differential equation
\begin{align} \label{eq:xdot=f(x)}
\dot{x}(\tau)=\sum_{(i,j)\in\mc{R}} \left(\kappa_{ij}\prod_{s=1}^n x_s(\tau) ^{B_{si}}\right)\cdot(B_{\cdot j}-B_{\cdot i})
\end{align}
with state space $\RNNn$. We will use another form of \eqref{eq:xdot=f(x)} in the sequel. For this aim, we next introduce the matrix $I_{\kappa} \in \Rcc$ and the function $\Theta : \RNNn \to \RNNc$. Let us extend the function $\kappa : \mc{R} \to \RPL$ to $\mc{C} \times \mc{C}$ such that $\kappa_{ij}=0$ for $(i,j)\in(\mc{C}\times\mc{C})\bs\mc{R}$ (we do not make any distinction in notation between the original function and its extension). Define the matrix $I_{\kappa}\in\Rcc$ by
\begin{align}\label{eq:Ikappa}
  I_{\kappa}=
  \left[
  \begin{array}{cccc}
  \kappa_{11}-\sum_{i=1}^c \kappa_{1i} & \kappa_{21} & \cdots & \kappa_{c1} \\
  \kappa_{12} & \kappa_{22}-\sum_{i=1}^c \kappa_{2i} & \cdots & \kappa_{c2} \\
  \vdots & \vdots & \ddots & \vdots \\
  \kappa_{1c} & \kappa_{2c} & \cdots & \kappa_{cc}-\sum_{i=1}^c \kappa_{ci}
  \end{array}
  \right] \in\Rcc
\end{align}
and the function $\Theta : \RNNn \to \RNNc$ by
\begin{align*}
  \Theta(x)=
  \left[
  \begin{array}{c}
   \prod_{s=1}^nx_s^{B_{s1}} \\
   \prod_{s=1}^nx_s^{B_{s2}} \\
   \vdots \\
   \prod_{s=1}^nx_s^{B_{sc}}
  \end{array}
  \right]~~(x\in\RNNn).
\end{align*}
Note that both the rows and the columns of $I_{\kappa}$ and the coordinate functions of $\Theta$ correspond to the vertices of the directed graph $(\mc{C},\mc{R})$, i.e., to the complexes. With these notations, \eqref{eq:xdot=f(x)} can be written equivalently as
\begin{align*}
\dot{x}(\tau) = B \cdot I_{\kappa} \cdot \Theta(x(\tau)).
\end{align*}

The main object we are interested in in this paper is the \emph{set of positive steady states} of a mass action system, denoted by $E^{\kappa}_+$. Let us define $E^{\kappa}_+$ by
\begin{align*}
E^{\kappa}_+ = \{ x \in \RPLn ~|~ B \cdot I_{\kappa} \cdot \Theta(x)=0 \}.
\end{align*}
We have included the symbol $\kappa$ in the notation $E^{\kappa}_+$, because our primary goal in this paper is to investigate the dependence of the non-emptiness of the set of positive steady states on the rate coefficients.

We use the usual terminology of the theory of directed graphs in the sequel. We have also collected the required ones in Appendix \ref{sec:app_directed_graphs}. Denote by $\ell$ the number of weak components of the directed graph $(\mc{C},\mc{R})$. The weak components of $(\mc{C},\mc{R})$ are called \emph{linkage classes} in CRNT. In case all the linkage classes of $(\mc{C},\mc{R})$ are strongly connected, the network is called \emph{weakly reversible} in CRNT. Denote by $t$ the number of absorbing strong components of the directed graph $(\mc{C},\mc{R})$. Since each weak component contains at least one absorbing strong component, we have $\ell \leq t$.

A useful fact is that if $\ell = t$ then $\ran I_{\kappa}$ does \emph{not} depend on $\kappa$ (see e.g. \cite[Corollary 2.6]{boros:2012}). Based on this, we define the \emph{deficiency}, denoted by $\delta$, of a reaction network of the type $\ell = t$ by $\delta = \dim(\ker B \cap \ran I_{\kappa})$.

The following theorem is a classical result in CRNT, it is called the \emph{Deficiency-Zero Theorem}, see e.g. \cite[Theorem 6.1.1]{feinberg:1987}.

\begin{theorem}\label{thm:dfc0,t=l=1}
Let $(\mc{X}, \mc{C}, \mc{R})$ be a chemical reaction network with $\ell = t = 1$ and $\delta = 0$. Then the following statements hold.
\begin{itemize}
  \item[(a)]  If $(\mc{C}, \mc{R})$ is strongly connected then for all $\kappa : \mc{R} \to \RPL$ we have $E^{\kappa}_+ \neq \emptyset$.
  \item[(b)] If $(\mc{C}, \mc{R})$ is \emph{not} strongly connected then for all $\kappa : \mc{R} \to \RPL$ we have $E^{\kappa}_+ = \emptyset$.
\end{itemize}
\end{theorem}

The also classical \emph{Deficiency-One Theorem} has been augmented recently, see \cite[Theorem 6.2.1]{feinberg:1987} and \cite[Theorem 3.19]{boros:2012}. To state the theorem, we need some further notations, which will then be used throughout this paper. For a chemical reaction network with $\ell=t=1$, denote by $\mc{C}'$ the set of complexes, which are in the terminal strong linkage class of $(\mc{C},\mc{R})$ and let $\mc{C}''=\mc{C}\backslash\mc{C}'$. Let $c'=|\mc{C}'|$ and $c''=|\mc{C}''|$ (thus, $c'' = c - c'$). With this, $I_{\kappa} \in \Rcc$, $\Theta : \Rn \to \R^c$, and any vector $v \in \R^c$ can be considered the block forms
\begin{align*}
I_{\kappa}=\left[
  \begin{array}{cc}
     I'_{\kappa}  & *      \\
     0            & I''_{\kappa} \\
  \end{array}
  \right] \in \R^{(c' + c'') \times (c' + c'')},~~
\Theta=
\left[\begin{array}{c}
\Theta' \\
\Theta''
\end{array}\right]:\Rn\to\R^{c'+c''},~~\mbox{and}~~v=
\left[\begin{array}{c}
v' \\
v''
\end{array}\right]\in\R^{c'+c''},
\end{align*}
where $I'_{\kappa} \in \R^{c' \times c'}$, $\Theta':\Rn\to\R^{c'}$, and $v'\in\R^{c'}$ correspond to the complexes in $\mc{C}'$, while $I''_{\kappa} \in \R^{c'' \times c''}$, $\Theta'':\Rn\to\R^{c''}$, and $v''\in\R^{c''}$ correspond to the complexes in $\mc{C}''$. There are several ways to prove that $I_{\kappa}''$ is invertible (provided that $\mc{C}'' \neq \emptyset$), see e.g. \cite[Lemma 2.5]{boros:2012}.

Primarily, we are interested in this paper in mass action systems for which $\ell = t = 1$ and $\delta = 1$ hold. For such systems, the linear subspace $\ker B \cap \ran I_{\kappa}$ is one-dimensional and does not depend on $\kappa$ (recall that for systems with $\ell = t$ we have $\delta = \dim(\ker B \cap \ran I_{\kappa})$). For systems that are moreover \emph{not} weakly reversible, let us fix $h \in \Rc$ such that
\begin{align}\label{eq:fix_h}
0\neq h \in \ker B \cap \ran I_{\kappa} ~ \mbox{and} ~ h(\mc{C}'') \leq 0,
\end{align}
where the notation $h(\mc{C}'')$ is understood in accordance with \eqref{eq:p(X_0)=sump(x)}, i.e., the sum of certain coordinates of $h$ is non-positive, where the summation goes for those complexes that are out of the absorbing strong component of $(\mc{C}, \mc{R})$.

With these notations in hand, we are ready to state the Deficiency-One Theorem in its augmented form (see \cite[Theorem III.7]{boros:2010} and \cite[Theorems 3.11 and 3.19]{boros:2012}).

\begin{theorem}\label{thm:dfc1,t=l=1}
Let $(\mc{X}, \mc{C}, \mc{R})$ be a chemical reaction network with $\ell = t = 1$ and $\delta = 1$. Then the following statements hold.
\begin{itemize}
  \item[(a)]  If $(\mc{C}, \mc{R})$ is strongly connected then for all $\kappa : \mc{R} \to \RPL$ we have $E^{\kappa}_+ \neq \emptyset$.
  \item[(b)] Assume that $(\mc{C}, \mc{R})$ is \emph{not} strongly connected. Let $h \in \Rc$ be as in \eqref{eq:fix_h} and fix $\kappa : \mc{R} \to \RPL$. Then
              \begin{align*}
              E^{\kappa}_+ \neq \emptyset ~ \mbox{if and only if all the coordinates of} ~ (I_{\kappa}'')^{-1}h'' ~ \mbox{are positive}.
              \end{align*}
\end{itemize}
\end{theorem}

Seemingly, there is some freedom in the choice of $h$ (it is chosen from a one-dimensional linear subspace of $\Rc$). On the one hand, if $h(\mc{C}'') < 0$ then $h$ is determined up to a positive scalar multiplier. However, it is clear that the choice of this positive scalar multiplier does not affect the condition $(I_{\kappa}'')^{-1}h'' \in \RPL^{c''}$. On the other hand, if $h(\mc{C}'') = 0$ then $h$ is determined up to a nonzero scalar multiplier. However, it is easy to see that $(I_{\kappa}'')^{-1}h'' \in \RPL^{c''}$ implies $h(\mc{C}'') < 0$ (indeed, let $\vt'' = (I_{\kappa}'')^{-1}h''$ and take the sum of the coordinates on both sides of $I_{\kappa}''\vt'' = h''$). Thus, again, this nonzero scalar multiplier does not affect the condition $(I_{\kappa}'')^{-1}h'' \in \RPL^{c''}$.

As a side remark, we also mention that both in Theorems \ref{thm:dfc0,t=l=1} and \ref{thm:dfc1,t=l=1}, once $E^{\kappa}_+ \neq$ holds, there is exactly one positive steady state in each positive stoichiometric class.

Thus, if the reaction network satisfies $\ell = t = 1$ and $\delta = 0$ then the non-emptiness of $E^{\kappa}_+$ does not depend on $\kappa$. Also, if the reaction network satisfies $\ell = t = 1$ and $\delta = 1$, and moreover $(\mc{C}, \mc{R})$ is strongly connected then, again, the non-emptiness of $E^{\kappa}_+$ does not depend on $\kappa$. However, by Theorem \ref{thm:dfc1,t=l=1} $(b)$, we have a different situation for mass action systems for which the underlying reaction network satisfies $\ell = t = 1$ and $\delta = 1$, but $(\mc{C}, \mc{R})$ is \emph{not} strongly connected. For these mass action systems, the non-emptiness of $E^{\kappa}_+$ \emph{may} depends on $\kappa$. We used the word \lq\lq may'', because for such mass action systems three different kind of phenomena can occur:
\begin{itemize}

  \item $E^{\kappa}_+ \neq \emptyset$ for all $\kappa$ (i.e., for all $\kappa$ all the coordinates of $(I_{\kappa}'')^{-1}h''$ are positive),

  \item $E^{\kappa}_+ = \emptyset$ for all $\kappa$ (i.e., for all $\kappa$ there exists a non-positive coordinate of $(I_{\kappa}'')^{-1}h''$), and

  \item the non-emptiness of $E^{\kappa}_+$ depends on $\kappa$ (i.e., there exists $\kappa$ such that all the coordinates of $(I_{\kappa}'')^{-1}h''$ are positive and there also exists $\kappa$ such that there exists a non-positive coordinate of $(I_{\kappa}'')^{-1}h''$).

\end{itemize}

It is demonstrated in \cite[Analysis of Examples 3.7, 3.8, and 3.9]{boros:2012} that all of these three phenomena can indeed occur. The aim of the present paper is to provide characterisations of the above cases. Namely, we will formulate equivalent conditions to the statements
\begin{align}
\label{eq:exists}&\mbox{\lq\lq there exists $\kappa:\mc{R}\to\RPL$ such that $E^{\kappa}_+\neq\emptyset$'' and}\\
\label{eq:forall}&\mbox{\lq\lq for all $\kappa:\mc{R}\to\RPL$ we have $E^{\kappa}_+\neq\emptyset$''}.
\end{align}

In Section \ref{sec:specialcases} we examine the above questions under some extra assumptions on $(\mc{C},\mc{R})$. In Subsection \ref{subsec:chain} we will assume that $(\mc{C},\mc{R})$ is a \lq\lq chain''. As a generalisation of the results of Subsection \ref{subsec:chain}, we will assume in Subsection \ref{subsec:tree-like} that $(\mc{C},\mc{R})$ is \lq\lq tree-like''. As a matter of fact, we will obtain a recursive formula for the coordinates of $(I_{\kappa}'')^{-1}h''$ (the matrix $I_{\kappa}''$ has some special properties in these cases, which makes it possible to handle the computation of its inverse). Based on the obtained recursive formula, we will deduce equivalent conditions both for \eqref{eq:exists} and \eqref{eq:forall}. In Sections \ref{sec:general_exists} and \ref{sec:general_forall} we will assume only that $(\mc{C}, \mc{R})$ satisfies $\ell = t = 1$, but is \emph{not} strongly connected. Under this assumption, we provide equivalent conditions to \eqref{eq:exists} and \eqref{eq:forall} for these general cases in Sections \ref{sec:general_exists} and \ref{sec:general_forall}, respectively.

\section{Special cases}\label{sec:specialcases}

We always assume in the rest of this paper that the reaction network under consideration satisfies $\ell = t = 1$ and $\delta = 1$, but is \emph{not} strongly connected. Thus, $\mc{C}'' \neq \emptyset$. Since we will apply Theorem \ref{thm:dfc1,t=l=1} $(b)$, we fix $h \in \Rc$ as in \eqref{eq:fix_h} (recall that $\ell = t$ implies that $\ran I_{\kappa}$ does not depend on $\kappa$, and hence, $h$ is not influenced by $\kappa$). Also, let $\vt \in \R^c$ be such that
\begin{align}\label{eq:fix_vt}
I_{\kappa} \vt = h.
\end{align}
Thus, $\vt$ depends on $\kappa$. Since $I_{\kappa}$ is block upper triangular, we obtain that $I''_{\kappa} \vt'' = h''$. Thus, $\vt'' = (I_{\kappa}'')^{-1}h''$. By Theorem \ref{thm:dfc1,t=l=1} $(b)$, we have
\begin{align}\label{eq:EPL_NONEMPTY_iff_VT_IS_POS}
\mbox{$E^{\kappa}_+ \neq \emptyset$ if and only if $\vartheta'' \in \RPL^{c''}$.}
\end{align}
Thus, our aim in this section is to obtain a formula for the coordinates of $\vt''$.

We formulate a purely graph theoretical lemma, which will be useful in Subsections \ref{subsec:chain} and \ref{subsec:tree-like}. The required graph theoretical notions can be found in Appendix \ref{sec:app_directed_graphs}.

\begin{lemma}\label{lemma:excessIkappa}
Let $(V,A)$ be a directed graph and let $\vt : V \to \R$ be any function. Let $\kappa : V \times V \to \RNN$ be a function for which $\kappa_{ij}>0$ if and only if $(i,j)\in A$. Define $z : A \to \R$ by $z_{ij}=\kappa_{ij}\vt_i$ $((i,j)\in A)$, let $I_\kappa$ be as in \eqref{eq:Ikappa}, and let $h = I_{\kappa}\vt$. Then
\begin{align*}
\excess_z(U)=\sum_{j \in U} h_j ~ \mbox{for all} ~ U \subseteq V.
\end{align*}
\end{lemma}
\begin{proof}
Fix $j \in V$. Then, by the definitions of the matrix $I_{\kappa}$ and the $\excess$ function we obtain
\begin{align*}
h_j = \sum_{i \in V} (I_{\kappa})_{ji} \vt_i = \sum_{i \in V \bs \{j\}} \kappa_{ij}\vt_i - \sum_{i \in V \bs \{j\}} \kappa_{ji}\vt_j = z(\varrho^{\inarc}(j)) - z(\varrho^{\outarc}(j)) = \excess_z(j).
\end{align*}
With this, the statement of the lemma follows from \eqref{eq:excess_additive}.
\end{proof}

\subsection{The case $(\mc{C}, \mc{R})$ is a \lq\lq chain''} \label{subsec:chain}

Assume for this subsection that the graph of complexes (i.e., $(\mc{C},\mc{R})$) takes the special form
\begin{align}\label{eq:xypic_chain}
\xymatrix{
 C_c\ar@<.5ex>[rr]^{\kappa_{c,c-1}}                                         & &
 C_{c-1} \ar@<.5ex>[rr]^{\kappa_{c-1,c-2}} \ar@<.5ex>[ll]^{\kappa_{c-1,c}}  & &
 \cdots \ar@<.5ex>[ll]^{\kappa_{c-2,c-1}} \ar@<.5ex>[rr]^{\kappa_{43}}     & &
 C_3 \ar@<.5ex>[ll]^{\kappa_{34}} \ar@<.5ex>[rr]^{\kappa_{32}}            & &
 C_2 \ar@<.5ex>[ll]^{\kappa_{23}} \ar@<0ex>[rr]^{\kappa_{21}}             & &
 C_1,}
\end{align}
where we also indicated the rate coefficients. Note that in this case $\mc{C}' = \{C_1\}$ and $\mc{C}'' = \{C_2,\ldots,C_c\}$, while the matrix $I_{\kappa}''$ is tridiagonal and moreover
\begin{align}\label{eq:colsums=0_exceptforC2}
\mbox{the column sums of $I_{\kappa}''$ are zero except the one corresponding to $C_2$.}
\end{align}
Clearly, \eqref{eq:colsums=0_exceptforC2} is a consequence of the fact that leaving the set $\mc{C}''$ is only possible through $C_2$ (i.e., using the notations introduced in Appendix \ref{sec:app_directed_graphs}, $\emptyset \neq \varrho^{\outarc}(\mc{C}'') \subseteq \varrho^{\outarc}(C_2)$). Proposition \ref{prop:chain_recursion} below provides a recursive formula for the coordinates of $\vartheta''$.

\begin{proposition}\label{prop:chain_recursion}
Let $(\mc{X},\mc{C},\mc{R},\kappa)$ be a deficiency-one mass action system for which $(\mc{C},\mc{R})$ takes the form \eqref{eq:xypic_chain}. Let $h$ and $\vt$ be as in \eqref{eq:fix_h} and \eqref{eq:fix_vt}, respectively. Then
\begin{align}
\label{eq:chain_vartheta2} \vartheta_2&=-\frac{1}{\kappa_{21}}\sum_{i=2}^c h_i ~ \mbox{and}\\
\label{eq:chain_varthetaj} \vartheta_j&=\frac{\kappa_{j-1,j}}{\kappa_{j,j-1}}\vartheta_{j-1} - \frac{1}{\kappa_{j,j-1}}\sum_{i=j}^c h_i~~(j=3,\ldots,c).
\end{align}
\end{proposition}
\begin{proof}
Application of Lemma \ref{lemma:excessIkappa} with $U = \mc{C}''$ yields $-\kappa_{21}\vartheta_2 = \sum_{i=2}^c h_i$. This proves \eqref{eq:chain_vartheta2}.

Fix $j \in \{3, \ldots, c\}$. Application of Lemma \ref{lemma:excessIkappa} with $U=\{C_j,\ldots,C_c\}$ yields
\begin{align*}
\kappa_{j-1,j}\vartheta_{j-1} - \kappa_{j,j-1}\vartheta_j = \sum_{i=j}^c h_i.
\end{align*}
This proves \eqref{eq:chain_varthetaj}.
\end{proof}

As a corollary of Proposition \ref{prop:chain_recursion}, we obtain directly a characterisation of those reaction networks in Proposition \ref{prop:chain_recursion} for which there exist rate coefficients such that the resulting mass action system has a positive steady state. Similarly, a characterisation is given for those networks for which the resulting mass action system has a positive steady state regardless of the values of the rate coefficients.

\begin{corollary}\label{cor:chain_exists}
Let $(\mc{X},\mc{C},\mc{R})$ be a deficiency-one reaction network for which $(\mc{C},\mc{R})$ takes the form \eqref{eq:xypic_chain}. Let $h \in \R^c$ be as in \eqref{eq:fix_h}. Then there exists $\kappa:\mc{R}\to\RPL$ such that $E^{\kappa}_+\neq\emptyset$ if and only if
\begin{align*}
\sum_{i=2}^c h_i<0.
\end{align*}
\end{corollary}
\begin{proof}
The statement follows directly from \eqref{eq:EPL_NONEMPTY_iff_VT_IS_POS} and Proposition \ref{prop:chain_recursion}.
\end{proof}

\begin{corollary}\label{cor:chain_forall}
Let $(\mc{X},\mc{C},\mc{R})$ be a deficiency-one reaction network for which $(\mc{C},\mc{R})$ takes the form \eqref{eq:xypic_chain}. Let $h \in \R^c$ be as in \eqref{eq:fix_h}. Then for all $\kappa:\mc{R}\to\RPL$ we have $E^{\kappa}_+\neq\emptyset$ if and only if
\begin{align*}
\begin{split}
&\sum_{i=2}^c h_i<0 ~ \mbox{and} \\
&\sum_{i=j}^c h_i \leq 0 ~ \mbox{for all} ~ j \in \{3, \ldots, c\}.
\end{split}
\end{align*}
\end{corollary}
\begin{proof}
The statement follows directly from \eqref{eq:EPL_NONEMPTY_iff_VT_IS_POS} and Proposition \ref{prop:chain_recursion}.
\end{proof}

\subsection{The case $(\mc{C}, \mc{R})$ is \lq\lq tree-like''} \label{subsec:tree-like}

In this subsection, we generalise the results obtained in Subsection \ref{subsec:chain}. We will not pose in this subsection the condition that the graph of complexes $(\mc{C},\mc{R})$ takes the form \eqref{eq:xypic_chain}, rather we assume that $(\mc{C},\mc{R})$ satisfies
\begin{align}
\label{eq:tree-like_1} &\ell=t=1 ~ \mbox{and} ~ (\mc{C},\mc{R}) ~ \mbox{is not strongly connected},\\
\label{eq:tree-like_2} &\mbox{there exists a unique}~l\in\mc{C}''~\mbox{such that}~\varrho^{\outarc}(l)\cap\varrho^{\outarc}(\mc{C}'')\neq\emptyset,~\mbox{and}\\
\label{eq:tree-like_3} &\mbox{for all $i\in\mc{C}''$ there exists a unique directed path from $i$ to $l$.}
\end{align}

Thus, by \eqref{eq:tree-like_2}, we have $\emptyset \neq \varrho^{\outarc}(\mc{C}'') \subseteq \varrho^{\outarc}(l)$. Consider that the graph of complexes $(\mc{C},\mc{R})$ takes the form
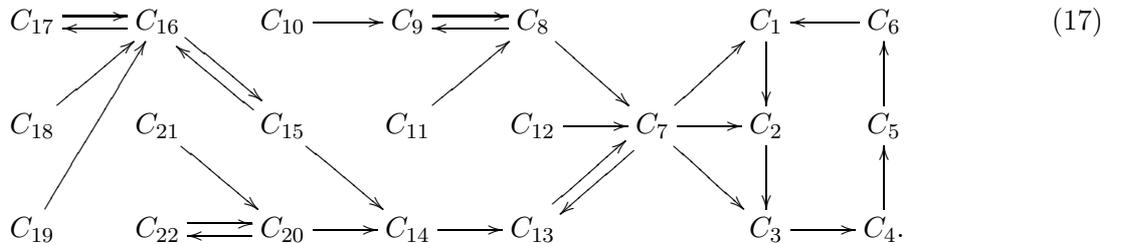
\begin{align}\label{eq:xypic_tree-like}
\xymatrix{
C_{17}\ar@<.5ex>[r] & C_{16}\ar@<.5ex>[rd]
                       \ar@<.5ex>[l]       & C_{10}\ar[r]         & C_9\ar@<.5ex>[r] & C_8\ar[rd]
                                                                                     \ar@<.5ex>[l]     &                   & C_1\ar[d] & C_6\ar[l] \\
C_{18}\ar[ru]  & C_{21}\ar[rd]        & C_{15}\ar[rd]
                                              \ar@<.5ex>[lu] & C_{11}\ar[ru]    & C_{12}\ar[r]         & C_7\ar[ru]
                                                                                                            \ar[r]
                                                                                                            \ar[rd]
                                                                                                            \ar@<.5ex>[ld] & C_2\ar[d] & C_5\ar[u] \\
C_{19}\ar[ruu] & C_{22}\ar@<.5ex>[r]  & C_{20}\ar[r]
                                              \ar@<.5ex>[l]  & C_{14}\ar[r]     & C_{13}\ar@<.5ex>[ru] &                  & C_3\ar[r] & C_4.\ar[u] \\
}
\end{align}
Note that $\mc{C}' = \{C_1,\ldots,C_6\}$ and $\mc{C}'' = \{C_7, \ldots, C_{22} \}$ for \eqref{eq:xypic_tree-like}. Also, \eqref{eq:xypic_tree-like} satisfies \eqref{eq:tree-like_1}, \eqref{eq:tree-like_2}, and \eqref{eq:tree-like_3} with $l = C_7$.

Clearly, the graph in \eqref{eq:xypic_chain} satisfies \eqref{eq:tree-like_1}, \eqref{eq:tree-like_2}, and \eqref{eq:tree-like_3} with $l = C_2$. Thus, the results of this subsection are indeed generalisations of the results of Subsection \ref{subsec:chain}.

For $j\in \mc{C}''$ denote by $P_j$ the unique directed path from $j$ to $l$ and for $i \in \mc{C}''$ define $U(i) \subseteq \mc{C}''$, called the \emph{set of descendants} of $i$, by
\begin{align*}
U(i) = \{j \in \mc{C}'' ~|~ i \in V[P_j] \},
\end{align*}
i.e., we collect those vertices for which the unique directed path to $l$ traverses $i$. Also, for $j \in \mc{C}'' \bs \{l\}$ define $p(j) \in V[P_j]$, called the \emph{parent} of $j$, by the implicit definition
\begin{align*}
\len(P_{p(j)}) = \len(P_j) - 1,
\end{align*}
i.e., the parent of $j$ is the second vertex on the unique directed path from $j$ to $l$. For example, for \eqref{eq:xypic_tree-like} we have $p(C_{20}) = C_{14}$ and $U(C_{20}) = \{C_{20},C_{21},C_{22}\}$.

We are now ready to state and prove the generalisation of Proposition \ref{prop:chain_recursion}.

\begin{proposition}\label{prop:tree-like_recursion}
Let $(\mc{X},\mc{C},\mc{R},\kappa)$ be a deficiency-one mass action system for which $(\mc{C},\mc{R})$ satisfies \eqref{eq:tree-like_1}, \eqref{eq:tree-like_2}, and \eqref{eq:tree-like_3}. Let $h$ and $\vt$ be as in \eqref{eq:fix_h} and \eqref{eq:fix_vt}, respectively. Then
\begin{align}
\label{eq:tree-like_varthetal}\vartheta_l&=-\frac{1}{\sum_{l'\in\mc{C}'}\kappa_{l,l'}}\sum_{i\in \mc{C}''} h_i,\\
\label{eq:tree-like_varthetaj}\vartheta_j&=\frac{\kappa_{p(j),j}}{\kappa_{j,p(j)}}\vartheta_{p(j)} - \frac{1}{\kappa_{j,p(j)}}\sum_{i\in U(j)} h_i~~(j\in \mc{C}''\bs\{l\}).
\end{align}
\end{proposition}
\begin{proof}
Application of Lemma \ref{lemma:excessIkappa} with $U = \mc{C}''$ yields $-\sum_{l'\in\mc{C}'}\kappa_{l,l'}\vartheta_l = \sum_{i \in \mc{C}''} h_i$. This proves \eqref{eq:tree-like_varthetal}.

Fix $j \in \mc{C}'' \bs \{l\}$. Application of Lemma \ref{lemma:excessIkappa} with $U=U(j)$ yields
\begin{align*}
\kappa_{p(j),j}\vartheta_{p(j)} - \kappa_{j,p(j)}\vartheta_j = \sum_{i \in U(j)} h_i.
\end{align*}
This proves \eqref{eq:tree-like_varthetaj}.
\end{proof}

As consequences of Proposition \ref{prop:tree-like_recursion}, we obtain directly the generalisations of Corollaries \ref{cor:chain_exists} and \ref{cor:chain_forall}.

\begin{corollary}\label{cor:tree-like_exists}
Let $(\mc{X},\mc{C},\mc{R})$ be a deficiency-one reaction network for which $(\mc{C},\mc{R})$ satisfies \eqref{eq:tree-like_1}, \eqref{eq:tree-like_2}, and \eqref{eq:tree-like_3}. Let $h \in \R^c$ be as in \eqref{eq:fix_h}. Then there exists $\kappa:\mc{R}\to\RPL$ such that $E^{\kappa}_+\neq\emptyset$ if and only if
\begin{align*}
\begin{split}
&\sum_{i\in \mc{C}''} h_i < 0~\mbox{and} \\
&\sum_{i\in U(j)} h_i < 0~\mbox{for all}~j\in\mc{C}''\bs\{l\}~\mbox{that satisfies}~\kappa_{p(j),j}=0.
\end{split}
\end{align*}
\end{corollary}
\begin{proof}
The statement follows directly from \eqref{eq:EPL_NONEMPTY_iff_VT_IS_POS} and Proposition \ref{prop:tree-like_recursion}.
\end{proof}

\begin{corollary}\label{cor:tree-like_forall}
Let $(\mc{X},\mc{C},\mc{R})$ be a deficiency-one reaction network for which $(\mc{C},\mc{R})$ satisfies \eqref{eq:tree-like_1}, \eqref{eq:tree-like_2}, and \eqref{eq:tree-like_3}. Let $h \in \R^c$ be as in \eqref{eq:fix_h}. Then for all $\kappa:\mc{R}\to\RPL$ we have $E^{\kappa}_+\neq\emptyset$ if and only if
\begin{align*}
\begin{split}
&\sum_{i\in \mc{C}''} h_i<0, \\
&\sum_{i\in U(j)} h_i < 0~\mbox{for all}~j\in\mc{C}''\bs\{l\}~\mbox{that satisfies}~\kappa_{p(j),j}=0,~\mbox{and} \\
&\sum_{i\in U(j)} h_i \leq 0~\mbox{for all}~j\in\mc{C}''\bs\{l\}~\mbox{that satisfies}~\kappa_{p(j),j}>0.
\end{split}
\end{align*}
\end{corollary}
\begin{proof}
The statement follows directly from \eqref{eq:EPL_NONEMPTY_iff_VT_IS_POS} and Proposition \ref{prop:tree-like_recursion}.
\end{proof}

Note that for \eqref{eq:xypic_tree-like} we have
\begin{align*}
\kappa_{p(j),j} &= 0 ~ \mbox{for} ~ j \in \{C_8, C_{10}, C_{11}, C_{12}, C_{14}, C_{15}, C_{18}, C_{19}, C_{20}, C_{21}\} ~ \mbox{and} \\
\kappa_{p(j),j} &> 0 ~ \mbox{for} ~ j \in \{C_9, C_{13}, C_{16}, C_{17}, C_{22}\}.
\end{align*}
Thus, application of Corollary \ref{cor:tree-like_exists} yields for \eqref{eq:xypic_tree-like} that there exists $\kappa: \mc{R} \to \RPL$ such that $E_+^{\kappa} \neq \emptyset$ if and only if
\begin{align}
\begin{split}\label{eq:tree-like_example_exists}
&h_7 + \cdots + h_{22} < 0,     h_8 + \cdots + h_{11} < 0,  h_{10} < 0,  h_{11} < 0,                    h_{12} < 0,  h_{14} + \cdots + h_{22}, \\
&h_{15} + \cdots + h_{19} < 0,  h_{18} < 0,                 h_{19} < 0,  h_{20} + h_{21} + h_{22} < 0,  ~\mbox{and}~ h_{21} < 0.
\end{split}
\end{align}
Similarly, application of Corollary \ref{cor:tree-like_forall} yields for \eqref{eq:xypic_tree-like} that for all $\kappa: \mc{R} \to \RPL$ we have $E_+^{\kappa} \neq \emptyset$ if and only if
\begin{align*}
&\mbox{\eqref{eq:tree-like_example_exists} holds and moreover} \\
&h_9 + h_{10} \leq 0,     h_{13} + \cdots + h_{22} \leq 0,  h_{16} + \cdots + h_{19} \leq 0,  h_{17} \leq 0, h_{22} \leq 0.
\end{align*}

\section{The existence of rate coefficients such that the set of positive steady states is nonempty}\label{sec:general_exists}

In this section we generalise Corollary \ref{cor:tree-like_exists}. We will assume only that the reaction network under consideration is of deficiency-one and satisfies \eqref{eq:tree-like_1}. The main tool we use is the following purely graph theoretical theorem. We provide its proof in Appendix \ref{sec:app_positive_h-transshipment}. The graph theoretical notions and notations required to understand the statement of Theorem \ref{thm:exists_positive_h-transshipment} are summarized in Appendix \ref{sec:app_directed_graphs}.

\begin{theorem}\label{thm:exists_positive_h-transshipment}
Let $(V,A)$ be a weakly connected directed graph and let $h : V \to \R$ be a function with $h(V) = 0$. Then there exists a function $z : A \to \RPL$ such that $\excess_z = h$ if and only if
\begin{align}\label{eq:h(U)<0}
h(U) < 0 ~ \mbox{for all} ~ \emptyset \neq U \subsetneq V ~ \mbox{with} ~ \varrho^{\inarc}(U) = \emptyset.
\end{align}
\end{theorem}

As a corollary of Theorem \ref{thm:exists_positive_h-transshipment}, we obtain the generalisation of Corollary \ref{cor:tree-like_exists}.

\begin{corollary}\label{cor:general_exists}
Let $(\mc{X},\mc{C},\mc{R})$ be a reaction network for which $\ell = t = 1$ and $\delta = 1$. Assume that $(\mc{C},\mc{R})$ is not strongly connected and let $h \in \R^c$ be as in \eqref{eq:fix_h}. Then there exists $\kappa: \mc{R} \to \RPL$ such that $E_+^{\kappa} \neq \emptyset$ if and only if
\begin{align}\label{eq:h(U)<0_CRN}
\mbox{$h(\wt{\mc{C}}) < 0$ for all $\emptyset \neq \wt{\mc{C}} \subsetneq \mc{C}$ such that $\varrho^{\inarc}(\wt{\mc{C}})=\emptyset$.}
\end{align}
\end{corollary}
\begin{proof}
To prove that \eqref{eq:h(U)<0_CRN} is necessary, let $\kappa: \mc{R} \to \RPL$ be such that $E_+^{\kappa} \neq \emptyset$. Then, by Theorem \ref{thm:dfc1,t=l=1} $(b)$, we have $(I_{\kappa}'')^{-1}h'' \in \RPL^{c''}$. Let $\vt \in \R^c$ be such that $I_{\kappa} \vt = h$ and define $z : \mc{R} \to \R$ by $z_{ij} = \kappa_{ij} \vt_i$. Since $\vt'' = (I_{\kappa}'')^{-1}h''$, we have $\vt''_i > 0$ for all $i \in \mc{C}''$. Also, let $\emptyset \neq \wt{\mc{C}} \subsetneq \mc{C}$ be such that $\varrho^{\inarc}(\wt{\mc{C}})=\emptyset$. Then clearly $\wt{\mc{C}} \subseteq \mc{C}''$ and $\varrho^{\outarc}(\wt{\mc{C}}) \neq \emptyset$. Thus, by Lemma \ref{lemma:excessIkappa}, we have
\begin{align*}
h(\wt{\mc{C}}) = \excess_z (\wt{\mc{C}}) = - z(\varrho^{\outarc}(\wt{\mc{C}})) = - \sum_{(i,j) \in \varrho^{\outarc}(\wt{\mc{C}})} \kappa_{ij} \vt_i < 0.
\end{align*}

To prove the sufficiency of \eqref{eq:h(U)<0_CRN}, let $\kappa : \mc{R} \to \RPL$ be such that $\excess_{\kappa} = h$ (see Theorem \ref{thm:exists_positive_h-transshipment}). Also, define $\vt \in \R^c$ by $\vt_i = 1$ $(i \in \mc{C})$. Then clearly $I_{\kappa} \vt = h$ (recall \eqref{eq:Ikappa}). Thus, $(I_{\kappa}'')^{-1}h'' = \vt'' \in \RPL^{c''}$ and therefore Theorem \ref{thm:dfc1,t=l=1} $(b)$ concludes the proof.
\end{proof}

Consider that the graph of complexes takes the form
\begin{align}\label{eq:xypic_exists}
\xymatrix{
                   & C_6 \ar@<.5ex>[r] & C_5 \ar@<.5ex>[l] \ar[r] & C_1 \ar[r] & C_4 \ar[d] \\
C_7 \ar[ru] \ar[r] & C_8 \ar[r]        & C_9 \ar[ru]              & C_2 \ar[u] & C_3. \ar[l] \\
}
\end{align}
Application of Corollary \ref{cor:general_exists} to a reaction network for which the graph of complexes is \eqref{eq:xypic_exists} yields that there exists $\kappa: \mc{R} \to \RPL$ such that $E_+^{\kappa} \neq \emptyset$ if and only if
\begin{align*}
&h_5 + h_6 + h_7 < 0, h_7 < 0, h_7 + h_8 < 0, h_7 + h_8 + h_9 < 0,\\
&h_5 + h_6 + h_7 + h_8 < 0, ~ \mbox{and} ~ h_5 + h_6 + h_7 + h_8 + h_9 <0.
\end{align*}

By Corollary \ref{cor:general_exists}, the sets of interest are $\emptyset \neq \wt{\mc{C}} \subsetneq \mc{C}$ for which $\varrho^{\inarc}(\wt{\mc{C}})=\emptyset$. Therefore, we conclude this section by some comments on these sets (see Proposition \ref{prop:R(i)} and Corollary \ref{cor:wt_mc_C} below). To this end, let us define for $i \in \mc{C}''$ the set $R(i) \subseteq \mc{C}''$ by
\begin{align}\label{eq:R(i)}
R(i) = \{j \in \mc{C}'' ~|~ \mbox{there exists a directed path from $j$ to $i$}\}.
\end{align}

\begin{proposition}\label{prop:R(i)}
Assume that $(\mc{C}, \mc{R})$ satisfies \eqref{eq:tree-like_1} and for $i \in \mc{C}''$ let $R(i)$ be as in \eqref{eq:R(i)}. Then
\begin{itemize}
  \item[(a)] for all $i \in \mc{C}''$ we have $i \in R(i)$,

  \item[(b)] for all $i \in \mc{C}''$ the set $R(i)$ is the disjoint union of some strong linkage classes,

  \item[(c)] if $i_1 \in \mc{C}''$ and $i_2 \in \mc{C}''$ are in the same strong linkage class then $R(i_1) = R(i_2)$,

  \item[(d)] for all $i \in \mc{C}''$ we have $\emptyset \neq R(i) \subsetneq \mc{C}$ and $\varrho^{\inarc}(R(i))=\emptyset$,

  \item[(e)] if $\emptyset \neq \wt{\mc{C}} \subsetneq \mc{C}$ is such that $\varrho^{\inarc}(\wt{\mc{C}})=\emptyset$ and $i \in \wt{\mc{C}}$ then $R(i) \subseteq \wt{\mc{C}}$, and

  \item[(f)] for all $i_1, i_2 \in \mc{C}''$ we have $\emptyset \neq R(i_1) \cup R(i_2) \subsetneq \mc{C}$ and $\varrho^{\inarc}(R(i_1) \cup R(i_2))=\emptyset$.
\end{itemize}
\end{proposition}
\begin{proof}
All the statements are trivial.
\end{proof}

\begin{corollary}\label{cor:wt_mc_C}
Assume that $(\mc{C}, \mc{R})$ satisfies \eqref{eq:tree-like_1} and for $i \in \mc{C}''$ let $R(i)$ be as in \eqref{eq:R(i)}. Let $\emptyset \neq \wt{\mc{C}} \subsetneq \mc{C}$ be such that $\varrho^{\inarc}(\wt{\mc{C}})=\emptyset$. Then there exists $J \subseteq \mc{C}''$ such that $\wt{\mc{C}} = \cup_{i \in J} R(i)$.
\end{corollary}
\begin{proof}
The statement follows directly from Proposition \ref{prop:R(i)} $(a)$ and $(e)$.
\end{proof}

For \eqref{eq:xypic_exists} we have
\begin{align*}
R(5) = R(6) = \{5, 6, 7\}, R(7) = \{7\}, R(8) = \{7, 8\}, ~ \mbox{and} ~ R(9) = \{7, 8, 9\}.
\end{align*}
It can be seen that the sets $\emptyset \neq \wt{\mc{C}} \subsetneq \mc{C}$ for which $\varrho^{\inarc}(\wt{\mc{C}})=\emptyset$ holds for \eqref{eq:xypic_exists} are exactly the sets
\begin{align*}
R(5), R(7), R(8), R(9), R(5) \cup R(8), ~ \mbox{and} ~ R(5) \cup R(9).
\end{align*}


\section{The non-emptiness of the set of positive steady states regardless of the values of the rate coefficients}\label{sec:general_forall}

We generalised Corollary \ref{cor:tree-like_exists} in Section \ref{sec:general_exists}. Our aim in the present section is to generalise Corollary \ref{cor:tree-like_forall}. Namely, to provide an equivalent condition to the statement \lq\lq for all $\kappa:\mc{R}\to\RPL$ we have $E^{\kappa}_+\neq\emptyset$'' for deficiency-one reaction networks for which the graph of complexes satisfies \eqref{eq:tree-like_1}. By Theorem \ref{thm:dfc1,t=l=1} $(b)$, the important object is $(I_{\kappa}'')^{-1}h''$, thus it does not restrict the generality if we contract the absorbing strong component of $(\mc{C}, \mc{R})$ into one vertex. So assume throughout this section that
\begin{align*}
\mc{C}' ~ \mbox{is a singleton.}
\end{align*}
Also, in order to ease the notation, we identify the set $\mc{C}'$ with the sole element in that set. (Though it is straightforward to extend all the definitions, proofs, and results of this section to the case $|\mc{C}'| \geq 2$, we still suppose $|\mc{C}'| = 1$ in order to avoid unnecessary technical complications.)

The rest of this section is organised as follows. After providing a formula for the entries of $(I_{\kappa}'')^{-1}$ (see Theorem \ref{thm:inverse_Ikappa''}), we prove the generalisation of Corollary \ref{cor:tree-like_forall} (see Corollary \ref{cor:forallkappa_hU<=0}). As it will be demonstrated on \eqref{eq:xypic_U}, the obtained result contains certain redundancies. We will get rid of these redundancies in three steps (see Corollaries \ref{cor:forallkappa_hU(i)<=0}, \ref{cor:mcJ}, \ref{cor:equiv_without_redundancy}). After some further manipulation, we arrive to Theorem \ref{thm:main}, which is the main result of this paper. Finally, we provide some additional results related to the condition that appears in Theorem \ref{thm:main} (see Propositions \ref{prop:laminar_mcU} and \ref{prop:equiv_Wj_subset_C''j}).

We first provide a formula for the entries of the inverse of $I_{\kappa}''$ via the Matrix-Tree Theorem. This formula in some other context and with slightly different notations also appears in \cite[Appendix]{leighton:rivest:1983}. Please refer to Appendix \ref{sec:app_mtx-tree-thm} for more on the Matrix-Tree Theorem.

\begin{theorem}\label{thm:inverse_Ikappa''}
Assume that $(\mc{C}, \mc{R})$ satisfies \eqref{eq:tree-like_1}. Fix $i, j \in \mc{C}''$. Then
\begin{align}\label{eq:Ikappainverse}
((I_{\kappa}'')^{-1})_{ji} = -\frac{\sum_{\wt{\mc{R}} \in \mc{T}_{\mc{D}}^{ij}(\mc{C}' \cup \{j\})}\kappa_{\wt{\mc{R}}}}
                                           {\sum_{\wt{\mc{R}} \in \mc{T}_{\mc{D}}(\mc{C}')}\kappa_{\wt{\mc{R}}}},
\end{align}
where the summation in the enumerator goes for those $(\mc{C}' \cup \{j\})$-branchings $\wt{\mc{R}}$ in $\mc{D} = (\mc{C}, \mc{R})$ for which there exists a directed path from $i$ to $j$ in $(\mc{C}, \wt{\mc{R}})$, while the summation in the denominator goes for the $\mc{C}'$-branchings $\wt{\mc{R}}$ in $\mc{D} = (\mc{C}, \mc{R})$ (see Appendix \ref{sec:app_directed_graphs} for the definitions of these standard graph theoretical terms). The symbol $\kappa_{\wt{\mc{R}}}$ is a shorthand notation for the product $\prod_{a \in \wt{\mc{R}}} \kappa_a$.
\end{theorem}
\begin{proof}
Application of Theorem \ref{thm:matrixtree} to the transpose of $I_{\kappa}$ twice yield
\begin{align*}
((I_{\kappa}'')^{-1})_{ji} = ((I_{\kappa}''^{\top})^{-1})_{ij}
                                    &= \frac{(-1)^{i+j} d_{ji}(I_{\kappa}''^{\top})}{\det I_{\kappa}''^{\top}} =
                                       \frac{(-1)^{i+j} d_{\mc{C}' \cup \{j\}, \mc{C}' \cup \{i\}}(I_{\kappa}^{\top})}{d_{\mc{C}', \mc{C}'}(I_{\kappa}^{\top})}= \\
                                    &= \frac{(-1)^{c-|\mc{C}'|-1}\sum_{\wt{\mc{R}} \in \mc{T}_{\mc{D}}^{ij}(\mc{C}' \cup \{j\})}\kappa_{\wt{\mc{R}}}}
                                            {(-1)^{c-|\mc{C}'|}\sum_{\wt{\mc{R}} \in \mc{T}_{\mc{D}}(\mc{C}')}\kappa_{\wt{\mc{R}}}},
\end{align*}
where $d_{Q_1, Q_2}(Z)$ denotes the determinant of that matrix, which is obtained from $Z$ by deleting the rows with index in $Q_1$ and the columns with index in $Q_2$.
\end{proof}

See Figure \ref{fig:branching_mtxtree} for an illustration of the notions appearing in Theorem \ref{thm:inverse_Ikappa''}.
\begin{figure}[h!t!b!]
  \begin{align*}
  \xymatrix{
  C_4 \ar@<.5ex>[d]
      \ar@/^1pc/[dd] & & C_4 \ar[d]
                         ^{\kappa_{43}} & C_4 \ar@/^1pc/[dd]
                                          ^{\kappa_{42}}     & C_4 \ar@/^1pc/[dd]
                                                               ^{\kappa_{42}}     & & C_4            & C_4            & C_4            & C_4                & C_4                   \\
  C_3 \ar@<.5ex>[d]
      \ar@<.5ex>[u]  & & C_3 \ar[d]
                         ^{\kappa_{32}} & C_3 \ar[d]
                                          _{\kappa_{32}}     & C_3 \ar[u]
                                                               ^{\kappa_{34}}     & & C_3 \ar[d]
                                                                                      ^{\kappa_{32}} & C_3 \ar[u]
                                                                                                       ^{\kappa_{34}} & C_3 \ar[u]
                                                                                                                        ^{\kappa_{34}} & C_3 \ar[d]
                                                                                                                                         ^{\kappa_{32}}     & C_3 \ar[u]
                                                                                                                                                              _{\kappa_{34}}         \\
  C_2 \ar[d]
      \ar@<.5ex>[u]
      \ar@/^1pc/[uu] & & C_2 \ar[d]
                         ^{\kappa_{21}} & C_2 \ar[d]
                                          ^{\kappa_{21}}     & C_2 \ar[d]
                                                               ^{\kappa_{21}}     & & C_2 \ar[d]
                                                                                      ^{\kappa_{21}} & C_2 \ar[d]
                                                                                                       ^{\kappa_{21}} & C_2 \ar[u]
                                                                                                                        ^{\kappa_{23}} & C_2 \ar@/^1pc/[uu]
                                                                                                                                         ^{\kappa_{24}} & C_2 \ar@/^1pc/[uu]
                                                                                                                                                          ^{\kappa_{24}}             \\
  C_1                & & C_1        & C_1                & C_1                & & C_1        & C_1        & C_1        & C_1                & C_1                \\
  }
  \end{align*}
  \caption{An example of a graph of complexes $\mc{D} = (\mc{C}, \mc{R})$ with $\mc{C} = \{C_1, C_2, C_3, C_4\}$ and $\mc{C}' = \{C_1\}$ (on the left), the three $1$-branchings in $\mc{D}$ (in the middle), and the five $\{1, 4\}$-branchings in $\mc{D}$ (on the right). Note that $|\mc{T}^{24}_{\mc{D}}(\{1,4\})| = 3$, $|\mc{T}^{34}_{\mc{D}}(\{1,4\})| = 4$, and $|\mc{T}^{44}_{\mc{D}}(\{1,4\})| = 5$. Thus, e.g. for $i=2$ and $j=4$ the numerator in \eqref{eq:Ikappainverse} is the sum of $3$ products and each of these products has $2$ factors, namely, $\kappa_{23}\kappa_{34} + \kappa_{24}\kappa_{32} + \kappa_{24}\kappa_{34}$.}
  \label{fig:branching_mtxtree}
\end{figure}
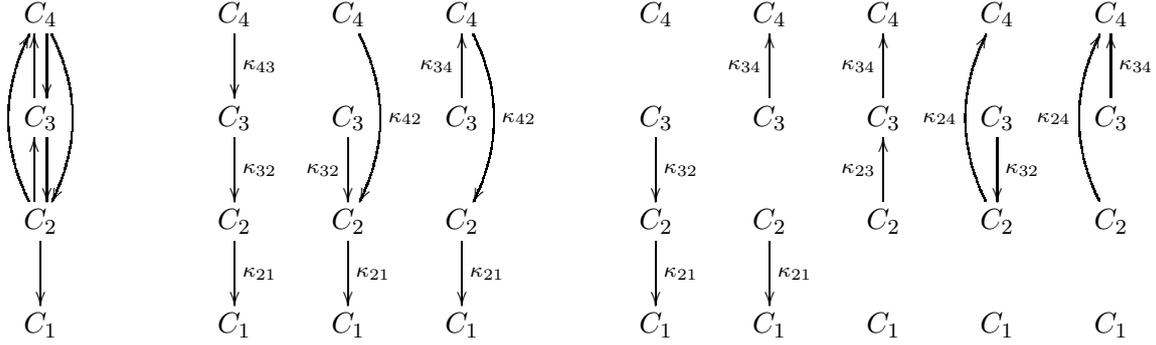

Let us introduce the notation $L_{\mc{D}}(\kappa) = \sum_{\wt{\mc{R}} \in \mc{T}_{\mc{D}}(\mc{C}')}\kappa_{\wt{\mc{R}}}$. Thus, e.g. for the example in Figure \ref{fig:branching_mtxtree} we have
\begin{align*}
L_{\mc{D}}(\kappa) = \kappa_{21}\kappa_{32}\kappa_{43} + \kappa_{21}\kappa_{32}\kappa_{42} + \kappa_{21}\kappa_{34}\kappa_{42}.
\end{align*}
From this point on, let $h$ and $\vt$ be as in \eqref{eq:fix_h} and \eqref{eq:fix_vt}, respectively. As a consequence of Theorem \ref{thm:inverse_Ikappa''}, for $j\in \mc{C}''$ we have
\begin{align*}
\vt_j = \left[(I_{\kappa}'')^{-1}h''\right]_j &= -\frac{1}{L_{\mc{D}}(\kappa)}\sum_{i\in \mc{C}''}h_i\left(\sum_{\wt{\mc{R}}\in\mc{T}^{ij}_{\mc{D}}(\mc{C}' \cup \{j\})}\kappa_{\wt{\mc{R}}}\right)=\\
&= -\frac{1}{L_{\mc{D}}(\kappa)} \sum_{\wt{\mc{R}} \in \mc{T}_{\mc{D}}(\mc{C}' \cup \{j\})}\left(\kappa_{\wt{\mc{R}}} \cdot h(V[\wt{\mc{R}},j])\right),
\end{align*}
where $V[\wt{\mc{R}},j]$ denotes the vertex set of the $j$-arborescence of the $(\mc{C}' \cup \{j\})$-branching $\wt{\mc{R}}$ (see Appendix \ref{sec:app_directed_graphs}). Note that for a vertex $j \in \mc{C}''$ and a set $U \subseteq \mc{C}''$ there exists a $(\mc{C}' \cup \{j\})$-branching $\wt{\mc{R}}$ such that $V[\wt{\mc{R}},j] = U$ if and only if
\begin{align}
\label{eq:j-inarb_1}&\mbox{for all $i'\in U$ there exists $P \in \orarr{i',j}$ such that $V[P] \subseteq U$ and}\\
\label{eq:j-inarb_2}&\mbox{for all $i'\in \mc{C}''\bs U$ there exists $P \in \orarr{i',\mc{C}'}$ such that $V[P] \subseteq \mc{C} \bs U$.}
\end{align}
Let us define the set $\mc{C}_{j-\inarb}''$ by
\begin{align}\label{eq:def_of_Cjinarb}
\mc{C}_{j-\inarb}'' = \{ U \subseteq \mc{C}'' ~|~ U ~\mbox{satisfies \eqref{eq:j-inarb_1} and \eqref{eq:j-inarb_2}}\}.
\end{align}
With this, we have
\begin{align*}
\vt_j = -\frac{1}{L_{\mc{D}}(\kappa)}\sum_{U \in \mc{C}''_{j-\inarb}} \left(h(U)\sum_{\substack{\wt{\mc{R}}\in\mc{T}_{\mc{D}}(\mc{C}' \cup \{j\})\\
V[\wt{\mc{R}},j]=U}}\kappa_{\wt{\mc{R}}}\right).
\end{align*}

\begin{corollary}\label{cor:forallkappa_hU<=0}
Let $(\mc{X},\mc{C},\mc{R})$ be a reaction network for which $\ell = t = 1$ and $\delta = 1$. Assume that $(\mc{C},\mc{R})$ is not strongly connected and let $h \in \R^c$ be as in \eqref{eq:fix_h}. Then for all $\kappa:\mc{R}\to\RPL$ we have $E^{\kappa}_+\neq\emptyset$ if and only if
\begin{align}\label{eq:equiv_EPL_nonempty_1}
\mbox{for all $j\in\mc{C}''$ and} \begin{cases} \mbox{for all $U\in\mc{C}_{j-\inarb}''$ we have $h(U)\leq0$ and} \\
                                                \mbox{there exists $U\in\mc{C}_{j-\inarb}''$ such that $h(U) < 0$.} \\\end{cases}
\end{align}
\end{corollary}
\begin{proof}
Since $L_{\mc{D}}(\kappa)$ is a positive number, it suffices to show that \eqref{eq:equiv_EPL_nonempty_1} is equivalent to
\begin{align}\label{eq:forall_j_<0}
\mbox{for all $\kappa:\mc{R}\to\RPL$ and for all $j \in \mc{C}''$ we have} ~ \sum_{U \in \mc{C}''_{j-\inarb}} \left(h(U)\sum_{\substack{\wt{\mc{R}}\in\mc{T}_{\mc{D}}(\mc{C}' \cup \{j\})\\
V[\wt{\mc{R}},j]=U}}\kappa_{\wt{\mc{R}}}\right) < 0.
\end{align}
It is obvious that \eqref{eq:equiv_EPL_nonempty_1} implies \eqref{eq:forall_j_<0}.

To prove the other direction, assume that \eqref{eq:forall_j_<0} holds and fix $j \in \mc{C}''$. It is clear that the case $h(U) = 0$ for all $U \in \mc{C}''_{j-\inarb}$ would contradict \eqref{eq:forall_j_<0}. Thus, it suffices to exclude that there exists $U \in \mc{C}''_{j-\inarb}$ such that $h(U) > 0$. Suppose by contradiction that $\ol{U} \in \mc{C}''_{j-\inarb}$ is such that $h(\ol{U}) > 0$. If there is no other element of $\mc{C}''_{j-\inarb}$ than $\ol{U}$ then we obviously get a contradiction with \eqref{eq:forall_j_<0}. So suppose for the rest of this proof that $\ol{U}$ is not the only element of $\mc{C}''_{j-\inarb}$. Then
\begin{align}\label{eq:olU+notolU}
\sum_{U \in \mc{C}''_{j-\inarb}} \left(h(U)\sum_{\substack{\wt{\mc{R}}\in\mc{T}_{\mc{D}}(\mc{C}' \cup \{j\})\\
V[\wt{\mc{R}},j]=U}}\kappa_{\wt{\mc{R}}}\right) =
h(\ol{U})\sum_{\substack{\wt{\mc{R}}\in\mc{T}_{\mc{D}}(\mc{C}' \cup \{j\})\\
V[\wt{\mc{R}},j]=\ol{U}}}\kappa_{\wt{\mc{R}}} + \sum_{U \in \mc{C}''_{j-\inarb} \bs \{\ol{U}\}} \left(h(U)\sum_{\substack{\wt{\mc{R}}\in\mc{T}_{\mc{D}}(\mc{C}' \cup \{j\})\\
V[\wt{\mc{R}},j]=U}}\kappa_{\wt{\mc{R}}}\right),
\end{align}
where the first term is positive and is not affected by $\kappa|_{\varrho^{\inarc}(\ol{U}) \cup \varrho^{\outarc}(\ol{U})}$, while for all $U \in \mc{C}''_{j-\inarb} \bs \{\ol{U}\}$ and for all $\wt{\mc{R}}\in\mc{T}_{\mc{D}}(\mc{C}' \cup \{j\})$ such that $V[\wt{\mc{R}},j]=U$, we have $\wt{\mc{R}} \cap (\varrho^{\inarc}(\ol{U}) \cup \varrho^{\outarc}(\ol{U})) \neq \emptyset$. Thus, by setting the values of $\kappa|_{\varrho^{\inarc}(\ol{U}) \cup \varrho^{\outarc}(\ol{U})}$ close enough to zero, we can achieve that the absolute value of the second term of the right hand side of \eqref{eq:olU+notolU} is smaller than the (positive) value of the first term in the same. This contradicts \eqref{eq:forall_j_<0}.
\end{proof}

Note that Corollary \ref{cor:forallkappa_hU<=0} is a generalisation of Corollary \ref{cor:tree-like_forall}. However, there are certain redundancies in the set of conditions in \eqref{eq:equiv_EPL_nonempty_1}, and therefore our aim in the rest of this section is to get rid of these. In order to illustrate the result of Corollary \ref{cor:forallkappa_hU<=0} and also the redundancy in \eqref{eq:equiv_EPL_nonempty_1}, consider that the graph of complexes takes the form

\begin{align}\label{eq:xypic_U}
\xymatrix{
                  &                   & C_4 \ar@<.5ex>[r] & C_3 \ar@<.5ex>[l]
                                                                \ar[rd]       &            &     \\
C_8 \ar@<.5ex>[r] & C_7 \ar@<.5ex>[l]
                        \ar[ru]
                        \ar[rd]       &                   &                   & C_2 \ar[r] & C_1. \\
                  &                   & C_6 \ar@<.5ex>[r]
                                            \ar[rru]      & C_5 \ar@<.5ex>[l]
                                                                \ar[ru]       &            &     \\
}
\end{align}
One can check by short calculation that
\begin{align}\label{eq:ex_U}
\begin{array}{crccccl}
\mc{C}''_{2-\inarb} = & \{ & \{2,3,4,5,6,7,8\} &             &             &             & \},              \\
\mc{C}''_{3-\inarb} = & \{ & \{3,4\},          & \{3,4,7,8\} &             &             & \},              \\
\mc{C}''_{4-\inarb} = & \{ & \{4\},            & \{3,4\},    & \{4,7,8\},  & \{3,4,7,8\} & \},              \\
\mc{C}''_{5-\inarb} = & \{ & \{5\},            & \{5,6\},    & \{5,6,7,8\} &             & \},              \\
\mc{C}''_{6-\inarb} = & \{ & \{6\},            & \{5,6\},    & \{6,7,8\},  & \{5,6,7,8\} & \},              \\
\mc{C}''_{7-\inarb} = & \{ & \{7,8\}           &             &             &             & \}, ~ \mbox{and} \\
\mc{C}''_{8-\inarb} = & \{ & \{8\},            & \{7,8\}     &             &             & \}.              \\
\end{array}
\end{align}
Note however that e.g. the set $\{3,4,7,8\} \in \mc{C}''_{3-\inarb}$ is the disjoint union of the sets $\{3,4\} \in \mc{C}''_{3-\inarb}$ and $\{7,8\} \in \mc{C}''_{7-\inarb}$. Hence, once we require that $h(\{3,4\}) \leq 0$ and $h(\{7,8\}) \leq 0$, it is unnecessary to require also $h(\{3,4,7,8\}) \leq 0$, because then it is automatically satisfied. Similarly, one may easily see for \eqref{eq:xypic_U} that for all $j \in \mc{C}''$ and for all $U \in \mc{C}''_{j-\inarb}$, the set $U$ is the disjoint union of some of the sets
\begin{align*}
\{2,3,4,5,6,7,8\}, \{3,4\}, \{4\}, \{5\}, \{6\}, \{7,8\}, ~ \mbox{and} ~ \{8\}.
\end{align*}
Moreover, such a partition of $U$ is unique. Thus, for \eqref{eq:xypic_U} the following are equivalent.
\begin{align}
\label{eq:equiv_ex_h(U)leq0}&\begin{array}{l}\mbox{For all $j\in\mc{C}''$ and for all $U\in\mc{C}_{j-\inarb}''$ we have $h(U)\leq0$.}\end{array} \\
\label{eq:equiv_ex_h(U(i))leq0}&
\begin{array}{l}
\mbox{We have $h(\{2,3,4,5,6,7,8\}) \leq 0$, $h(\{3,4\}) \leq 0$, $h(\{4\}) \leq 0$,}\\
\mbox{$h(\{5\}) \leq 0$, $h(\{6\}) \leq 0$, $h(\{7,8\}) \leq 0$, and $h(\{8\}) \leq 0$} \\
\mbox{(i.e., we require the non-positivity of $h$ only for the first sets in each row of \eqref{eq:ex_U}).}
\end{array}
\end{align}

We formulate in the following lemma that the above mentioned facts about \eqref{eq:xypic_U} hold generally.

\begin{lemma}\label{lemma:U_partition}
Assume that $(\mc{C}, \mc{R})$ satisfies \eqref{eq:tree-like_1}. For $j \in \mc{C}''$ let $\mc{C}''_{j-\inarb}$ be as in \eqref{eq:def_of_Cjinarb}. For $i \in \mc{C}$ let us define $U(i) \subseteq \mc{C}$ by
\begin{align}\label{eq:def_of_U(i)}
U(i)=\{k\in\mc{C} ~|~ \mbox{for all $P \in \orarr{k,\mc{C}'}$ we have $i \in V[P]$}\},
\end{align}
i.e., $k \in \mc{C}$ is an element of $U(i)$ if all the directed paths from $k$ to $\mc{C}'$ must traverse $i$ (clearly, for $i \in \mc{C}''$ we have $U(i) \subseteq \mc{C}''$). Then
\begin{itemize}
  \item[(a)] for all $j \in \mc{C}''$ we have $U(j) \in \mc{C}''_{j-\inarb}$ and
  \item[(b)] for all $j \in \mc{C}''$ and for all $U \in \mc{C}''_{j-\inarb}$ there exists a unique $\mc{I}_U \subseteq \mc{C}''$ such that
             \begin{align*}
             U = \sideset{}{^*}\bigcup_{i \in \mc{I}_U} U(i),
             \end{align*}
\end{itemize}
where the symbol $^*$ stresses that if $i, i' \in \mc{I}_U$ and $i \neq i'$ then $U(i) \cap U(i') =  \emptyset$.
\end{lemma}

The proof of Lemma \ref{lemma:U_partition} is carried out right after the proof of Lemma \ref{lemma:U(i)_abcd} below. Clearly, we have $U(\mc{C}') = \mc{C}$, but we will use the set $U(\mc{C}')$ only after Theorem \ref{thm:main}. Before that, we will be interested in the collection $\{U(i) ~|~ i \in \mc{C}''\}$. We remark that the notation $U(i)$ is in accordance with the similar one in Subsection \ref{subsec:tree-like}. Note that for the reaction network \eqref{eq:xypic_U} we have
\begin{align*}
&U(2) = \{2,3,4,5,6,7,8\}, U(3) = \{3,4\}, U(4) = \{4\}, U(5) = \{5\},\\
&U(6) = \{6\}, U(7) = \{7,8\}, ~ \mbox{and} ~ U(8) = \{8\}.
\end{align*}

Before we prove Lemma \ref{lemma:U_partition}, we explore some properties of the collection $\{U(i)~|~i\in\mc{C}''\}$ in the following lemma. Note that for all $k\in\mc{C}''$ there exists a directed path from $k$ to $\mc{C}'$ (i.e., for all $k \in \mc{C}''$ we have $\orarr{k,\mc{C}'} \neq \emptyset$).

\begin{lemma}\label{lemma:U(i)_abcd}
Assume that $(\mc{C}, \mc{R})$ satisfies \eqref{eq:tree-like_1}. Let $i, i' \in \mc{C}''$ and define $U(i)$ and $U(i')$ as in \eqref{eq:def_of_U(i)}. Then
\begin{itemize}
  \item[(a)] $i \in U(i)$,
  \item[(b)] if $i' \in U(i)\bs\{i\}$ then $U(i') \subseteq U(i) \bs \{i\}$,
  \item[(c)] if $i \neq i'$ then either $U(i) \subsetneq U(i')$ or $U(i) \supsetneq U(i')$ or $U(i) \cap U(i') = \emptyset$, and
  \item[(d)] $\varrho^{\outarc}(U(i))\subseteq \varrho^{\outarc}(i)$ (i.e., all the arcs that leave $U(i)$ have tail $i$).
\end{itemize}
\end{lemma}
\begin{proof}
Statement $(a)$ is trivial.

To prove statement $(b)$, assume that $i'\in U(i)\bs\{i\}$ and let $i''\in U(i')$. Then all the directed paths from $i''$ to $\mc{C}'$ traverse $i'$ and since $i'\in U(i)$, they must also traverse $i$. This proves that $U(i')\subseteq U(i)$. To prove $(b)$, it remains to show that $i\notin U(i')$. However, it is also obvious, because $i'\in U(i)$ guarantees that there exists a directed path from $i'$ to $\mc{C}'$, which traverses $i$. Since there cannot be vertex repetition in a directed path and $i\neq i'$, this also shows that there exists a directed path from $i$ to $\mc{C}'$ that does not traverse $i'$.

To show $(c)$, suppose that $U(i)\cap U(i')\neq\emptyset$ and let $i''\in U(i)\cap U(i')$. Then all the directed paths from $i''$ to $\mc{C}'$ must traverse both $i$ and $i'$. Note that the order of $i$ and $i'$ on these directed paths must be the same, otherwise we could easily construct two directed paths from $i''$ to $\mc{C}'$, one of which avoids $i$ and the other one avoids $i'$. As a consequence, either $i\in U(i')\bs\{i'\}$ or $i'\in U(i)\bs\{i\}$. In both cases we are done by $(b)$.

To prove $(d)$, suppose by contradiction that there exist $i'\in U(i)\bs\{i\}$ and $i''\in\mc{C}\bs U(i)$ such that $(i',i'')\in\mc{R}$. Then there exists $P\in\orarr{i'',\mc{C}'}$ such that $i \notin V[P]$. Therefore $\con(i',P)\in\orarr{i',\mc{C}'}$ is a directed path that avoids $i$, contradicting $i'\in U(i)$ (see Appendix \ref{sec:app_directed_graphs} for the definition of the concatenation).
\end{proof}

\begin{proofof}[Proof of Lemma \ref{lemma:U_partition}]
To prove $(a)$, fix $j \in \mc{C}''$ and let $i' \in \mc{C}''$. Assume first that $i' \in U(j)$ and let $P \in \orarr{i',\mc{C}'}$. By Lemma \ref{lemma:U(i)_abcd} $(d)$, $V[P^{i':j}] \subseteq U(j)$ ($P^{i':j}$ denotes the part of $P$ from $i'$ to $j$, see Appendix \ref{sec:app_directed_graphs}). Assume now that $i' \in \mc{C}'' \bs U(j)$. We need to show that there exists a directed path from $i'$ to $\mc{C}'$ that avoids $U(j)$. Suppose by contradiction that for all $P \in \orarr{i',\mc{C}'}$ we have $V[P] \cap U(j) \neq \emptyset$. Then, by Lemma \ref{lemma:U(i)_abcd} $(d)$, it follows that $j \in V[P]$, which contradicts $i' \notin U(j)$.

It is left to prove $(b)$. Fix $j \in \mc{C}''$ and $U \in \mc{C}''_{j-\inarb}$. For $i \in U$, we have $U(i) \subseteq U$, because otherwise there would be an element $i' \in U(i) \bs U$, for which there does not exist a directed path from $i'$ to $\mc{C}'$, which avoids $U$. Since we also have $i \in U(i)$ (see Lemma \ref{lemma:U(i)_abcd} $(a)$), it holds that $U = \cup_{i \in U} U(i)$. From this and Lemma \ref{lemma:U(i)_abcd} $(c)$ it follows that $U$ is indeed the disjoint union of some $U(i)$'s and this partition is unique. Namely those $U(i)$'s take part in the partition, which are maximal inside $U$.
\end{proofof}

We obtain the following corollary, which is the first step towards the simplification of \eqref{eq:equiv_EPL_nonempty_1}.

\begin{corollary}\label{cor:redundant_U_U(i)}
Assume that $(\mc{C}, \mc{R})$ satisfies \eqref{eq:tree-like_1} and let $h \in \R^c$ be arbitrary. For $i, j \in \mc{C}''$ let $\mc{C}''_{j-\inarb}$ be as in \eqref{eq:def_of_Cjinarb} and let $U(i)$ be as in \eqref{eq:def_of_U(i)}. Then the following two statements are equivalent.
\begin{itemize}
  \item[(A)] For all $j\in\mc{C}''$ and for all $U\in\mc{C}_{j-\inarb}''$ we have $h(U)\leq0$.
  \item[(B)] For all $i\in\mc{C}''$ we have $h(U(i))\leq0$.
\end{itemize}
\end{corollary}
\begin{proof}
Suppose that $(A)$ holds. Then $(B)$ follows directly from Lemma \ref{lemma:U_partition} $(a)$.

Now suppose that $(B)$ holds. Statement $(A)$ is then obtained immediately from Lemma \ref{lemma:U_partition} $(b)$.
\end{proof}

\begin{corollary}\label{cor:forallkappa_hU(i)<=0}
Let $(\mc{X},\mc{C},\mc{R})$ be a reaction network for which $\ell = t = 1$ and $\delta = 1$. Assume that $(\mc{C},\mc{R})$ is not strongly connected and let $h \in \R^c$ be as in \eqref{eq:fix_h}. For $i, j \in \mc{C}''$ let $\mc{C}''_{j-\inarb}$ be as in \eqref{eq:def_of_Cjinarb} and let $U(i)$ be as in \eqref{eq:def_of_U(i)}. Then for all $\kappa:\mc{R}\to\RPL$ we have $E^{\kappa}_+\neq\emptyset$ if and only if
\begin{align}\label{eq:equiv_EPL_nonempty_2}
\begin{cases}
\mbox{for all $i \in \mc{C}''$ we have $h(U(i))\leq0$ and} \\
\mbox{for all $j\in\mc{C}''$ there exists $i \in \cup_{U \in \mc{C}''_{j-\inarb}} \mc{I}_U$ such that $h(U(i)) < 0$,} \\
\end{cases}
\end{align}
where $\mc{I}_U$ is the unique subset of $\mc{C}''$ such that $U = \cup_{i' \in \mc{I}_U}^* U(i')$ (see Lemma \ref{lemma:U_partition} $(b)$.)
\end{corollary}
\begin{proof}
The equivalence is a direct consequence of Corollary \ref{cor:forallkappa_hU<=0}, Lemma \ref{lemma:U_partition} $(b)$, and Corollary \ref{cor:redundant_U_U(i)}.
\end{proof}

In order to ease the notation in \eqref{eq:equiv_EPL_nonempty_2}, we define for $j \in \mc{C}''$ the set $W(j)$ by
\begin{align}\label{eq:def_of_W(j)}
W(j) = \bigcup_{U \in \mc{C}''_{j-\inarb}} \mc{I}_U.
\end{align}
By Lemmas \ref{lemma:U_partition} $(b)$ and \ref{lemma:U(i)_abcd} $(c)$, for $i, j \in \mc{C}''$ and $U \in \mc{C}''_{j-\inarb}$ we have
\begin{align}\label{eq:IU_equiv}
i \in \mc{I}_U ~ \mbox{if and only if} ~ i \in U ~ \mbox{and there does not exist} ~ i' \in U \bs \{i\} ~ \mbox{such that} ~ i \in U(i').
\end{align}

To illustrate the result of Corollary \ref{cor:forallkappa_hU(i)<=0}, note that for the reaction network \eqref{eq:xypic_U} we have
\begin{align*}
&\{2,3,4,5,6,7,8\} = U(2),                                                                                                                            \\
&\{3,4\} = U(3),            \{3,4,7,8\} = U(3) \cup^* U(7),                                                                                           \\
&\{4\} = U(4),              \{3,4\} = U(3),                  \{4,7,8\} = U(4) \cup^* U(7),                \{3,4,7,8\} = U(3) \cup^* U(7),             \\
&\{5\} = U(5),              \{5,6\} = U(5) \cup^* U(6),      \{5,6,7,8\} = U(5) \cup^* U(6) \cup^* U(7),                                              \\
&\{6\} = U(6),              \{5,6\} = U(5) \cup^* U(6),      \{6,7,8\} = U(6) \cup^* U(7),                \{5,6,7,8\} = U(5) \cup^* U(6) \cup^* U(7), \\
&\{7,8\} = U(7), ~ \mbox{and}                                                                                                                         \\
&\{8\} = U(8),              \{7,8\} = U(7).                                                                                                           \\
\end{align*}
Thus,
\begin{align*}
\begin{array}{llll}
\mc{I}_{\{2,3,4,5,6,7,8\}} = \{2\},    &                                  &                                     &                                    \\
\mc{I}_{\{3,4\}} = \{3\},              & \mc{I}_{\{3,4,7,8\}} = \{3, 7\}, &                                     &                                    \\
\mc{I}_{\{4\}} = \{4\},                & \mc{I}_{\{3,4\}} = \{3\},        & \mc{I}_{\{4,7,8\}} = \{4, 7\},      &\mc{I}_{\{3,4,7,8\}} = \{3, 7\},    \\
\mc{I}_{\{5\}} = \{5\},                & \mc{I}_{\{5,6\}} = \{5, 6\},     & \mc{I}_{\{5,6,7,8\}} = \{5, 6, 7\}, &                                    \\
\mc{I}_{\{6\}} = \{6\},                & \mc{I}_{\{5,6\}} = \{5, 6\},     & \mc{I}_{\{6,7,8\}} = \{6, 7\},      &\mc{I}_{\{5,6,7,8\}} = \{5, 6, 7\}, \\
\mc{I}_{\{7,8\}} = \{7\}, ~ \mbox{and} &                                  &                                     &                                 \\
\mc{I}_{\{8\}} = \{8\},                & \mc{I}_{\{7,8\}} = \{7\},         &                                     &                                   \\
\end{array}
\end{align*}
and therefore
\begin{align}\label{eq:W(j)'s_for_xypic_U}
\begin{split}
&W(2) = \{2\}, W(3) = \{3, 7\}, W(4) = \{3, 4, 7\}, W(5) = \{5, 6, 7\},\\
&W(6) = \{5, 6, 7\}, W(7) = \{7\}, ~ \mbox{and} ~ W(8) = \{7, 8\}.
\end{split}
\end{align}
Since
\begin{align*}
W(7) \subseteq W(3), W(7) \subseteq W(4), W(7) \subseteq W(5), W(7) \subseteq W(6), ~ \mbox{and} ~ W(7) \subseteq W(8),
\end{align*}
once we require that there exist $i_2 \in W(2)$ and $i_7 \in W(7)$ such that $h(U(i_2)) < 0$ and $h(U(i_7)) < 0$ hold, it is unnecessary to require also e.g. that there exists $i_4 \in W(4)$ such that $h(U(i_4)) < 0$, because that is automatically satisfied. Our aim is to get rid of these sort of redundancies in Corollary \ref{cor:forallkappa_hU(i)<=0}. For this, we first take a closer look at the collection $\{W(j) ~|~ j \in \mc{C}''\}$ in Proposition \ref{prop:equiv_to_i_in_W(j)}. During the proof of that proposition, we will use the following corollary of Menger's Theorem.

\begin{theorem}\label{thm:menger_2}
Let $D = (V,A)$ be a directed graph and let $s, t$ be such that $(s,t) \notin A$. Then there exists $P_1, P_2 \in \orarr{s, t}$ such that $V[P_1] \cap V[P_2] = \{s,t\}$ if and only if for all $i \in V \bs \{s, t\}$ there exists $P \in \orarr{s, t}$ such that $i \notin V[P]$.
\end{theorem}
\begin{proof}
The result is a direct consequence of Theorem \ref{thm:menger}. See \cite[Section 9.1]{schrijver:2003} for more on Menger's Theorem.
\end{proof}

\begin{proposition}\label{prop:equiv_to_i_in_W(j)}
Assume that $(\mc{C}, \mc{R})$ satisfies \eqref{eq:tree-like_1}. Let $j \in \mc{C}''$ and let $W(j)$ be as in \eqref{eq:def_of_W(j)}. Then
\begin{align}\label{eq:equiv_W(j)}
W(j)=\{i\in\mc{C}''~|~\mbox{there exist $P_j\in\orarr{i,j}$ and $P_{\mc{C}'}\in\orarr{i,\mc{C}'}$ such that $V[P_j]\cap V[P_{\mc{C}'}]=\{i\}$}\}.
\end{align}
\end{proposition}
\begin{proof}
Denote by $Q(j)$ the set on the right hand side of \eqref{eq:equiv_W(j)}. We will show \eqref{eq:equiv_W(j)} by showing that both $W(j) \subseteq Q(j)$ and $W(j) \supseteq Q(j)$ hold.

First we prove that $W(j) \subseteq Q(j)$ holds. By \eqref{eq:def_of_W(j)} and \eqref{eq:IU_equiv}, we obtain for $i \in \mc{C}''$ that $i \in W(j)$ if and only if
\begin{align}\label{eq:i_in_W(j)_equiv}
\mbox{there exists $U \in \mc{C}''_{j-\inarb}$ such that $i \in U$ and for all $i' \in U \bs \{i\}$ we have $i \notin U(i')$.}
\end{align}
To obtain an equivalent description of $Q(j)$, we will apply Theorem \ref{thm:menger_2} for the directed graph
\begin{align*}
\wh{\mc{D}} = (\mc{C} \cup \{\mf{c}\}, \mc{R} \cup \{(j, \mf{c}), (\mc{C}', \mf{c})\}),
\end{align*}
where $\mf{c}$ is an auxiliary vertex. Thus, we have added the arcs $(j, \mf{c})$ and $(\mc{C}', \mf{c})$ to $\mc{R}$. Application of Theorem \ref{thm:menger_2} with $s = i \in \mc{C}'' \bs \{j\}$ and $t = \mf{c}$ yields that for $i \in \mc{C}'' \bs \{j\}$ that we have $i \in Q(j)$ if and only if
\begin{align}\label{eq:i_in_Q(j)_equiv}
\mbox{$\orarr{i,j} \neq \emptyset$ and for all $i' \in \mc{C}'' \bs \{i\}$ we have $i \notin U(i')$ or there exists $P \in \orarr{i,j}$ such that $i' \notin V[P]$,}
\end{align}
where the \lq\lq or'' is inclusive. Using \eqref{eq:j-inarb_1}, it is obvious that \eqref{eq:i_in_W(j)_equiv} implies \eqref{eq:i_in_Q(j)_equiv} for $i \in \mc{C}'' \bs \{j\}$. Thus, taking also into account that $j \in Q(j)$ holds obviously, we obtain the inclusion $W(j) \subseteq Q(j)$.

It is left to prove $W(j) \supseteq Q(j)$. Fix $i \in Q(j)$, $P_j \in \orarr{i,j}$, and $P_{\mc{C}'} \in \orarr{i,\mc{C}'}$ such that $V[P_j] \cap V[P_{\mc{C}'}] = \{i\}$. Let
\begin{align}
\label{eq:U} U = \{ i' \in \mc{C}'' ~|~ \mbox{there exists} ~ P \in \orarr{i',j} ~ \mbox{such that} ~ V[P] \cap V[P_{\mc{C}'}] \subseteq \{i\} \},
\end{align}
i.e., we collect those vertices, from which it is possible to reach $j$ without traversing $P_{\mc{C}'}$ except maybe in $i$. It is trivial that $U$ in \eqref{eq:U} satisfies \eqref{eq:j-inarb_1}. Also, it is easy to see that $U$ in \eqref{eq:U} fulfills \eqref{eq:j-inarb_2}. Thus, we have $U\in\mc{C}_{j-\inarb}''$. The only thing it is left to check is that $i \in \mc{I}_U$. Clearly, $i \in U$ and $(V[P_{\mc{C}'}] \bs \{i\}) \cap U = \emptyset$. Hence, there does not exist $i' \in U \bs \{i\}$ such that $i \in U(i')$. Thus, by \eqref{eq:IU_equiv}, we obtain that $i \in \mc{I}_U$. This concludes the proof of the inclusion $W(j) \supseteq Q(j)$.
\end{proof}

In case $j_1,j_2\in\mc{C}''$ are such that $W(j_1)\subseteq W(j_2)$, it is redundant in Corollary \ref{cor:forallkappa_hU(i)<=0} to require that there exists $i \in W(j_2)$ such that $h(U(i))<0$, because this already follows if we require the same for $j_1$ instead of $j_2$. In order to get rid of these kind of redundancies, we take a closer look at the collection $\{ W(j) ~ | ~ j \in \mc{C}'' \}$. The following lemma is the key.

\begin{lemma}\label{lemma:Wj1_subseteq_Wj2}
Assume that $(\mc{C}, \mc{R})$ satisfies \eqref{eq:tree-like_1}. For $j \in \mc{C}''$ let $W(j)$ be as in \eqref{eq:def_of_W(j)}. Fix $j_1, j_2 \in \mc{C}''$ such that $j_1 \in W(j_2)$. Then $W(j_1) \subseteq W(j_2)$.
\end{lemma}
\begin{proof}
Let $j_3 \in W(j_1)$. Our aim is to show that $j_3$ is also an element of $W(j_2)$. If $j_3 = j_2$ then $j_3 \in W(j_2)$ trivially holds, so let us assume for the rest of this proof that $j_3 \neq j_2$. Similarly to the proof of Proposition \ref{prop:equiv_to_i_in_W(j)}, we will apply Theorem \ref{thm:menger_2} to the directed graph
\begin{align*}
\wh{\mc{D}} = (\mc{C} \cup \{\mf{c}\}, \mc{R} \cup \{(j_2, \mf{c}), (\mc{C}', \mf{c})\}),
\end{align*}
where $\mf{c}$ is an auxiliary vertex. Let
\begin{align*}
&P_{j_2} \in \orarr{j_1, j_2} ~ \mbox{and} ~ P_{\mc{C}'} \in \orarr{j_1, \mc{C}'} ~ \mbox{be such that} ~ V[P_{j_2}] \cap V[P_{\mc{C}'}] = \{j_1\} ~ \mbox{and} \\
&Q_{j_1} \in \orarr{j_3, j_1} ~ \mbox{and} ~ Q_{\mc{C}'} \in \orarr{j_3, \mc{C}'} ~ \mbox{be such that} ~ V[Q_{j_1}] \cap V[Q_{\mc{C}'}] = \{j_3\}.
\end{align*}
Clearly, once we show that for all $i \in \mc{C} \bs \{j_3\}$ there exists a directed path from $j_3$ to $\mf{c}$ in $\wh{\mc{D}}$ that does not traverse $i$, we can draw the conclusion $j_3 \in W(j_2)$ by Theorem \ref{thm:menger_2}.

First let $i \in \mc{C} \bs V[Q_{\mc{C}'}]$. Then $\con(Q_{\mc{C}'}, \mf{c})$ is a directed path from $j_3$ to $\mf{c}$ in $\wh{\mc{D}}$ that does not traverse $i$.

It is left to treat the case $i \in V[Q_{\mc{C}'}] \bs \{j_3\}$. Then $i \notin V[P_{j_2}]$ or $i \notin V[P_{\mc{C}'}]$, where the \lq\lq or'' is inclusive. If $i \notin V[P_{j_2}]$ then $\con(Q_{j_1},P_{j_2},\mf{c})$ is a directed walk from $j_3$ to $\mf{c}$ in $\wh{\mc{D}}$ that does not traverse $i$. If $i \notin V[P_{\mc{C}'}]$ then $\con(Q_{j_1},P_{\mc{C}'},\mf{c})$ is a directed walk from $j_3$ to $\mf{c}$ in $\wh{\mc{D}}$ that does not traverse $i$. In both cases, one may easily construct the desired directed path from the directed walk.
\end{proof}

For $j \in \mc{C}''$, denote by $\mc{C}''(j)$ the vertex set of that strong component of $(\mc{C},\mc{R})$, which contains $j$. Thus, for \eqref{eq:xypic_U} we have
\begin{align*}
\mc{C}''(2) = \{2\}, \mc{C}''(3) = \mc{C}''(4) = \{3, 4\}, \mc{C}''(5) = \mc{C}''(6) = \{5, 6\}, ~ \mbox{and} ~ \mc{C}''(7) = \mc{C}''(8) = \{7, 8\}.
\end{align*}
For $j \in \mc{C}''$ we say that it is possible to leave $\mc{C}''(j)$ through $j$ if $\varrho^{\outarc}(j) \cap \varrho^{\outarc}(\mc{C}''(j))) \neq \emptyset$. For \eqref{eq:xypic_U}, this property holds with $j \in \{2, 3, 5, 6, 7\}$.

\begin{corollary}\label{cor:j'_in_C''(j)_then_j'_in_W(j)}
Assume that $(\mc{C}, \mc{R})$ satisfies \eqref{eq:tree-like_1}. For $j \in \mc{C}''$ let $W(j)$ be as in \eqref{eq:def_of_W(j)}. Let $j \in \mc{C}''$ be such that \lq\lq it is possible to leave $\mc{C}''(j)$ through $j$'' (i.e., $\varrho^{\outarc}(j) \cap \varrho^{\outarc}(\mc{C}''(j))) \neq \emptyset$). Then for all $j' \in \mc{C}''(j)$ we have $W(j) \subseteq W(j')$.
\end{corollary}
\begin{proof}
Fix $j' \in \mc{C}''(j)$. There exists a directed path from $j$ to $j'$ which uses only vertices in $\mc{C}''(j)$. On the other hand, since $\varrho^{\outarc}(j) \cap \varrho^{\outarc}(\mc{C}''(j)) \neq \emptyset$, it is possible to reach $\mc{C}'$ from $j$ using only vertices from $\mc{C} \bs (\mc{C}''(j) \bs \{j\})$. Hence, by Proposition \ref{prop:equiv_to_i_in_W(j)}, we have $j \in W(j')$. Lemma \ref{lemma:Wj1_subseteq_Wj2} concludes the proof.
\end{proof}

It is clear from the above corollary that if $j_1,j_2 \in \mc{C}''$ are such that $\mc{C}''(j_1)=\mc{C}''(j_2)$, $\varrho^{\outarc}(j_1) \cap \varrho^{\outarc}(\mc{C}''(j_1)) \neq \emptyset$, and $\varrho^{\outarc}(j_2) \cap \varrho^{\outarc}(\mc{C}''(j_2))) \neq \emptyset$ then $W(j_1) = W(j_2)$. For the reaction network \eqref{eq:xypic_U} we indeed have $W(5) = W(6)$.

Let $\mc{J}$ be a subset of $\mc{C}''$ for which
\begin{align}
\begin{split}\label{eq:J}
&\mbox{$\mc{J}$ contains precisely one element of each non-absorbing strong component of $(\mc{C},\mc{R})$ and} \\
&\mbox{for all $j \in \mc{J}$ we have $\varrho^{\outarc}(j) \cap \varrho^{\outarc}(\mc{C}''(j)) \neq \emptyset$.}
\end{split}
\end{align}

For \eqref{eq:xypic_U} we have two choices for $\mc{J}$. One is $\{2,3,5,7\}$, while the other one is $\{2,3,6,7\}$. To be concrete, let $\mc{J} = \{2,3,5,7\}$. Due to Corollary \ref{cor:j'_in_C''(j)_then_j'_in_W(j)}, we have
\begin{align*}
W(3) \subseteq W(4), W(5) \subseteq W(6), ~ \mbox{and} ~ W(7) \subseteq W(8),
\end{align*}
which is indeed the case (see \eqref{eq:W(j)'s_for_xypic_U}).

We have thus obtained the following corollary.
\begin{corollary}\label{cor:mcJ}
Let $(\mc{X},\mc{C},\mc{R})$ be a reaction network for which $\ell = t = 1$ and $\delta = 1$. Assume that $(\mc{C},\mc{R})$ is not strongly connected and let $h \in \R^c$ be as in \eqref{eq:fix_h}. For $i, j \in \mc{C}''$ let $U(i)$ and $W(j)$ be as in \eqref{eq:def_of_U(i)} and \eqref{eq:def_of_W(j)}, respectively. Also, let $\mc{J}$ be as in \eqref{eq:J}. Then for all $\kappa:\mc{R}\to\RPL$ we have $E^{\kappa}_+\neq\emptyset$ if and only if
\begin{align}
\begin{cases}
\mbox{for all $i \in \mc{C}''$ we have $h(U(i))\leq0$ and} \\
\mbox{for all $j \in \mc{J}$ there exists $i \in W(j)$ such that $h(U(i)) < 0$.} \\
\end{cases}
\end{align}
\end{corollary}
\begin{proof}
The equivalence follows directly from \eqref{eq:def_of_W(j)} and  Corollaries \ref{cor:forallkappa_hU(i)<=0} and \ref{cor:j'_in_C''(j)_then_j'_in_W(j)}.
\end{proof}

Since for \eqref{eq:xypic_U} we have made the choice $\mc{J} = \{2,3,5,7\}$, Corollary \ref{cor:mcJ} suggests that the sets of importance are $W(2)$, $W(3)$, $W(5)$, and $W(7)$. Recall that
\begin{align}\label{eq:W(j)_for_j_in_mc{J}}
W(2) = \{2\}, W(3) = \{3, 7\}, W(5) = \{5, 6, 7\}, ~ \mbox{and} ~ W(7) = \{7\}.
\end{align}
As $7 \in W(3)$ and $7 \in W(5)$, by Lemma \ref{lemma:Wj1_subseteq_Wj2} we have $W(7) \subseteq W(3)$ and $W(7) \subseteq W(5)$ (which is anyway obvious from \eqref{eq:W(j)_for_j_in_mc{J}}). Thus, still there is redundancy in Corollary \ref{cor:mcJ}. Elimination of this redundancy is formulated in the following corollary.

\begin{corollary}\label{cor:equiv_without_redundancy}
Let $(\mc{X},\mc{C},\mc{R})$ be a reaction network for which $\ell = t = 1$ and $\delta = 1$. Assume that $(\mc{C},\mc{R})$ is not strongly connected and let $h \in \R^c$ be as in \eqref{eq:fix_h}. For $i, j \in \mc{C}''$ let $U(i)$ and $W(j)$ be as in \eqref{eq:def_of_U(i)} and \eqref{eq:def_of_W(j)}, respectively. Also, let $\mc{J}$ be as in \eqref{eq:J}. Then for all $\kappa:\mc{R}\to\RPL$ we have $E^{\kappa}_+\neq\emptyset$ if and only if
\begin{align}\label{eq:equiv_EPL_nonempty_3}
\begin{cases}
\mbox{for all $i \in \mc{C}''$ we have $h(U(i))\leq0$ and} \\
\mbox{for all $j \in \mc{J}$ such that $W(j) \subseteq \mc{C}''(j)$, there exists $i \in W(j)$ such that $h(U(i)) < 0$.} \\
\end{cases}
\end{align}
\end{corollary}
\begin{proof}
If $W(j) \nsubseteq \mc{C}''(j)$ for some $j \in \mc{J}$ then for $j' \in W(j) \bs \mc{C}''(j)$ we have $W(j') \subseteq W(j)$ (see Lemma \ref{lemma:Wj1_subseteq_Wj2}). Denote by $j''$ the sole element of the singleton $\mc{J} \cap \mc{C}''(j')$. Then we have $W(j'') \subseteq W(j') \subseteq W(j)$ (see Corollary \ref{cor:j'_in_C''(j)_then_j'_in_W(j)}). Hence, the statement \lq\lq there exists $i \in W(j'')$ such that $h(U(i))<0$'' implies that \lq\lq there exists $i \in W(j)$ such that $h(U(i))<0$''. Thus, the result follows from Corollary \ref{cor:mcJ}.
\end{proof}

For \eqref{eq:xypic_U} we have
\begin{align}\label{eq:W_(n)subseteq_C''}
\begin{split}
W(2) = \{2\}     &\subseteq  \{2\}   = \mc{C}''(2), \\
W(3) = \{3,7\}   &\nsubseteq \{3,4\} = \mc{C}''(3), \\
W(5) = \{5,6,7\} &\nsubseteq \{5,6\} = \mc{C}''(5), ~ \mbox{and} ~ \\
W(7) = \{7\}     &\subseteq  \{7,8\} = \mc{C}''(7). \\
\end{split}
\end{align}

To simplify further the condition \eqref{eq:equiv_EPL_nonempty_3}, we examine in the following proposition the collection $\{ U(i) ~|~ i \in W(j) \}$, where $j \in \mc{C}''$ is such that $\varrho^{\outarc}(j) \cap \varrho^{\outarc}(\mc{C}''(j))) \neq \emptyset$ and $W(j) \subseteq \mc{C}''(j)$.

\begin{lemma}\label{lemma:U(C''(j))}
Assume that $(\mc{C}, \mc{R})$ satisfies \eqref{eq:tree-like_1}. For $i, j \in \mc{C}''$ let $U(i)$ and $W(j)$ be as in \eqref{eq:def_of_U(i)} and \eqref{eq:def_of_W(j)}, respectively. Fix $j \in \mc{C}''$ such that $\varrho^{\outarc}(j) \cap \varrho^{\outarc}(\mc{C}''(j)) \neq \emptyset$ and $W(j) \subseteq \mc{C}''(j)$. Then the sets $\{ U(i) ~|~ i \in W(j) \}$ are disjoint and $\cup_{i \in W(j)}^* U(i) = U(\mc{C}''(j))$, where
\begin{align}\label{eq:def_of_U(C''(j))}
U(\mc{C}''(j))=\{ k \in \mc{C}'' ~|~ \mbox{all directed paths from $k$ to $\mc{C}'$ traverse $\mc{C}''(j)$}\}.
\end{align}
\end{lemma}
\begin{proof}
First we prove that the sets $\{ U(i) ~|~ i \in W(j) \}$ are disjoint. Let $i_1, i_2 \in W(j)$ be such that $i_1 \neq i_2$. Due to Lemma \ref{lemma:U(i)_abcd} $(c)$, it suffices to show that none of $U(i_1)$ and $U(i_2)$ contains the other one. Suppose by contradiction that $U(i_2) \subseteq U(i_1) \bs \{i_1\}$. Let $P_j \in \orarr{i_2,j}$ and $P_{\mc{C}'} \in \orarr{i_2,\mc{C}'}$ be such that $V[P_j] \cap V[P_{\mc{C}'}] = \{i_2\}$ (see Proposition \ref{prop:equiv_to_i_in_W(j)}). Since $i_2 \in U(i_1)$ by our hypothesis ($i_2 \in U(i_2)$ by Lemma \ref{lemma:U(i)_abcd} $(a)$), we have $i_1 \in V[P_{\mc{C}'}]$. Since $i_1 \neq i_2$, we have $i_1 \notin V[P_j]$. Let $P \in \orarr{j,\mc{C}'}$ be such that $V[P] \cap \mc{C}''(j) = \{j\}$ (recall that $\varrho^{\outarc}(j) \cap \varrho^{\outarc}(\mc{C}''(j))) \neq \emptyset$). Then clearly $i_1 \notin V[P]$ (recall that $i_1 \in W(j) \subseteq \mc{C}''(j)$ and the case $i_1 = j$ can trivially be excluded). Thus, $\con(P_j,P) \in \orarr{i_2,\mc{C}'}$ and $i_1 \notin V[\con(P_j,P)]$, contradicting $i_2 \in U(i_1)$. This contradiction proves that the sets $\{ U(i) ~|~ i \in W(j) \}$ are indeed disjoint.

It is left to prove that $\cup_{i \in W(j)}^* U(i) = U(\mc{C}''(j))$. It is obvious that $U(\mc{C}''(j)) \in \mc{C}''_{j-\inarb}$ (one may prove this similarly to the proof Lemma \ref{lemma:U_partition} $(a)$). Hence, we have $\mc{I}_{U(\mc{C}''(j))} \subseteq W(j)$ (see \eqref{eq:def_of_W(j)}). Also, note that for $i \in W(j)$ we have $i \in \mc{C}''(j)$ (recall that have we assumed in the lemma that $W(j) \subseteq \mc{C}''(j)$). Thus, for $i \in W(j)$ we obviously have $U(i) \subseteq U(\mc{C}''(j))$. Therefore,
\begin{align*}
U(\mc{C}''(j)) = \sideset{}{^*}\bigcup_{i \in \mc{I}_{U(\mc{C}''(j))}} U(i) \subseteq \sideset{}{^*}\bigcup_{i \in W(j)} U(i) \subseteq U(\mc{C}''(j)).
\end{align*}
As a consequence, all the inclusions in the above chain are equality. This concludes the proof of $\cup_{i \in W(j)}^* U(i) = U(\mc{C}''(j))$.
\end{proof}

As a consequence, we obtain Theorem \ref{thm:main} below, which is the main result of this paper. Recall that for $i, j \in \mc{C}''$
\begin{itemize}
  \item $U(i) = \{k \in \mc{C}'' ~|~ \mbox{all directed paths from $k$ to $\mc{C}'$ traverse $i$}\}$,
  \item $\mc{C}''(j)$ denotes the vertex set of that absorbing strong component of $(\mc{C}, \mc{R})$ which contains $j$,
  \item $U(\mc{C}''(j)) = \{k \in \mc{C}'' ~|~ \mbox{all directed paths from $k$ to $\mc{C}'$ traverse $\mc{C}''(j)$}\}$, and
  \item $W(j) = \{k\in\mc{C}''~|~\mbox{there exist $P_j\in\orarr{k,j}$ and $P_{\mc{C}'}\in\orarr{k,\mc{C}'}$ such that $V[P_j]\cap V[P_{\mc{C}'}]=\{k\}$}\}$.
\end{itemize}
Also, recall that $\mc{J} \subseteq \mc{C}''$ is such that $\mc{J}$ contains precisely one element of each non-absorbing strong component of $(\mc{C}, \mc{R})$ and for all $j \in \mc{J}$ we have $\varrho^{\outarc}(j) \cap \varrho^{\outarc}(\mc{C}''(j)) \neq \emptyset$.

\begin{theorem}\label{thm:main}
Let $(\mc{X},\mc{C},\mc{R})$ be a reaction network for which $\ell = t = 1$ and $\delta = 1$. Assume that $(\mc{C},\mc{R})$ is not strongly connected and let $h \in \R^c$ be as in \eqref{eq:fix_h}. For $i, j \in \mc{C}''$ let $U(i)$, $W(j)$, and $U(\mc{C}''(j))$ be as in \eqref{eq:def_of_U(i)}, \eqref{eq:def_of_W(j)}, and \eqref{eq:def_of_U(C''(j))}, respectively. Also, let $\mc{J}$ be as in \eqref{eq:J}. Then for all $\kappa:\mc{R}\to\RPL$ we have $E^{\kappa}_+\neq\emptyset$ if and only if
\begin{align}
\begin{cases}
\mbox{for all $i \in \mc{C}''$ we have $h(U(i)) \leq 0$ and} \\
\mbox{for all $j \in \mc{J}$ such that $W(j) \subseteq \mc{C}''(j)$, we have $h(U(\mc{C}''(j))) < 0$.} \\
\end{cases}
\end{align}
\end{theorem}
\begin{proof}
The statement follows directly from Corollary \ref{cor:equiv_without_redundancy} and Lemma \ref{lemma:U(C''(j))}.
\end{proof}

Since for \eqref{eq:xypic_U} we have $U(\mc{C}''(2)) = U(2)$ and $U(\mc{C}''(7)) = U(7)$ we obtain that for all $\kappa:\mc{R}\to\RPL$ we have $E^{\kappa}_+\neq\emptyset$ if and only if
\begin{align*}
&h(\{2,3,4,5,6,7,8\}) < 0, h(\{3,4\}) \leq 0, h(\{4\}) \leq 0, h(\{5\}) \leq 0, \\
&h(\{6\}) \leq 0, h(\{7,8\}) < 0, ~ \mbox{and} ~ h(\{8\}) \leq 0.
\end{align*}

Since the condition $W(j) \subseteq \mc{C}''(j)$ appeared in our main result, the rest of this section is devoted to provide an equivalent (and more transparent) condition to that. After some preparations, we will arrive to this equivalent condition in Proposition \ref{prop:equiv_Wj_subset_C''j}. Let us start by defining the set $\mc{U}(j)$ for $j \in \mc{J}$ by
\begin{align}\label{eq:def_mcU(j)}
\mc{U}(j) = \bigcup_{j' \in \mc{C}''(j)} U(j').
\end{align}
Also, let $\mc{U}(\mc{C}') = U(\mc{C}')$ (thus, $\mc{U}(\mc{C}') = \mc{C}$). Note that for \eqref{eq:xypic_U} we have
\begin{align*}
\mc{U}(2) = \{2, 3, 4, 5, 6, 7, 8\}, \mc{U}(3) = \{3, 4\}, \mc{U}(5) = \{5, 6\}, ~ \mbox{and} ~ \mc{U}(7) = \{7, 8\}.
\end{align*}
The following proposition states that the collection $\{\mc{U}(j) ~|~ j \in \mc{J}\}$ has a similar property as the collection $\{U(i) ~|~ i \in \mc{C}''\}$ has.

\begin{proposition}\label{prop:laminar_mcU}
Assume that $(\mc{C}, \mc{R})$ satisfies \eqref{eq:tree-like_1}. Let $\mc{J}$ be as in \eqref{eq:J} and for $j \in \mc{J}$ let $\mc{U}(j)$ be as in \eqref{eq:def_mcU(j)}. Let $j_1, j_2 \in \mc{J}$ be such that $j_1 \neq j_2$. Then either $\mc{U}(j_1) \subsetneq \mc{U}(j_2)$ or $\mc{U}(j_1) \supsetneq \mc{U}(j_2)$ or $\mc{U}(j_1) \cap \mc{U}(j_2) = \emptyset$.
\end{proposition}
\begin{proof}
Suppose that $\mc{U}(j_1) \cap \mc{U}(j_2) \neq \emptyset$. Let $j_1' \in \mc{C}''(j_1)$ and $j_2' \in \mc{C}''(j_2)$ be such that $U(j_1') \cap U(j_2') \neq \emptyset$. Then, by Lemma \ref{lemma:U(i)_abcd} $(b)$ and $(iii)$, either $U(j_1') \subsetneq U(j_2') \bs \{j_2'\}$ or $U(j_1') \bs \{j_1'\} \supsetneq U(j_2')$. Clearly, the two cases are symmetric. Suppose for the rest of this proof that $U(j_1') \subsetneq U(j_2') \bs \{j_2'\}$. Since $j_1'$ and $j_2'$ are not in the same strong component, this has the consequence that $\mc{C}''(j_1) \subsetneq U(j_2')$. Thus, by Lemma \ref{lemma:U(i)_abcd} $(b)$, we have
\begin{align*}
\mc{U}(j_1) = \bigcup_{j' \in \mc{C}''(j_1)} U(j') \subseteq U(j_2') \bs \{j_2'\} \subsetneq U(j_2') \subseteq \bigcup_{j' \in \mc{C}''(j_2)} U(j') = \mc{U}(j_2),
\end{align*}
which concludes the proof.
\end{proof}

A collection $\mc{Q}$ of subsets of a set is called \emph{laminar} if for all $Q_1, Q_2 \in \mc{Q}$ we have
\begin{align*}
Q_1 \subseteq Q_2 ~ \mbox{or} ~ Q_1 \supseteq Q_2 ~ \mbox{or} ~ Q_1 \cap Q_2 = \emptyset.
\end{align*}
Thus, by Lemma \ref{lemma:U(i)_abcd} $(c)$ and Proposition \ref{prop:laminar_mcU} both the collections
\begin{align}\label{eq:laminar_U_mcU}
\{U(i) ~|~ i \in \mc{C}''\} \cup \{\mc{C}\} ~ \mbox{and} ~ \{\mc{U}(j) ~|~ j \in \mc{J}\} \cup \{\mc{C}\}
\end{align}
are laminar. It is straightforward to associate a branching to a laminar collection $\mc{Q}$ that consists of distinct sets. The vertex set of the branching is $\mc{Q}$ itself, while for $Q_1, Q_2 \in \mc{Q}$ the ordered pair $(Q_1,Q_2)$ is an arc if $Q_1 \subseteq Q_2$ and there does not exist $Q_3 \in \mc{Q}$ such that $Q_1 \subseteq Q_3 \subseteq Q_2$. Denote by $T$ and $\mc{T}$ the arborescences associated to the laminar collections in \eqref{eq:laminar_U_mcU}, respectively (since all the other sets of these two collections are contained in $\mc{C}$, the associated branching is actually an arborescence with root $\mc{C}$). We have depicted $T$ and $\mc{T}$ associated to \eqref{eq:xypic_U} in Figure \ref{fig:T_mcT_mfT}.

It is also straightforward to associate to a directed graph $D = (V,A)$ an acyclic directed graph, denoted by $\mf{T}_D$, in the following way. Denote by $V^1,\ldots,V^k$ the vertex sets of the strong components of $D$. The vertex set of $\mf{T}_D$ is then $\{V^1, \ldots, V^k\}$, while for $k_1, k_2 \in \{1, \ldots, k\}$ such that $k_1 \neq k_2$, the ordered pair $(V^{k_1}, V^{k_2})$ is an arc of $\mf{T}_D$ if $\varrho^{\outarc}_D(V^{k_1}) \cap \varrho^{\inarc}_D(V^{k_2}) \neq \emptyset$. We will simply denote by $\mf{T}$ the acyclic directed graph $\mf{T}_{(\mc{C},\mc{R})}$. Thus, the vertex set of $\mf{T}$ is $\{\mc{C}''(j) ~|~ j \in \mc{J}\} \cup \{\mc{C}'\}$. We have depicted $\mf{T}$ associated to \eqref{eq:xypic_U} in Figure \ref{fig:T_mcT_mfT}.

\begin{figure}[!h!t!b]
\begin{align*}
\xymatrix{
             & \mc{C}      & T            &               &                   & \mc{C}           & \mc{T}            & \mc{C}'                    & \mf{T}              \\
             & U(2) \ar[u] &              &               &                   & \mc{U}(2) \ar[u] &                   & \mc{C}''(2) \ar[u]         &                     \\
U(3) \ar[ru] & U(5) \ar[u] & U(6) \ar[lu] & U(7) \ar[llu] & \mc{U}(3) \ar[ru] & \mc{U}(5) \ar[u] & \mc{U}(7) \ar[lu] & \mc{C}''(3) \ar[u]         & \mc{C}''(5) \ar[lu] \\
U(4) \ar[u]  &             &              & U(8) \ar[u]   &                   &                  &                   & \mc{C}''(7) \ar[u] \ar[ru] &                     \\
}
\end{align*}
\caption{The arborescences $T$ and $\mc{T}$ associated to the laminar families \eqref{eq:laminar_U_mcU} and the acyclic directed graph $\mf{T}$ for \eqref{eq:xypic_U}.}
\label{fig:T_mcT_mfT}
\end{figure}
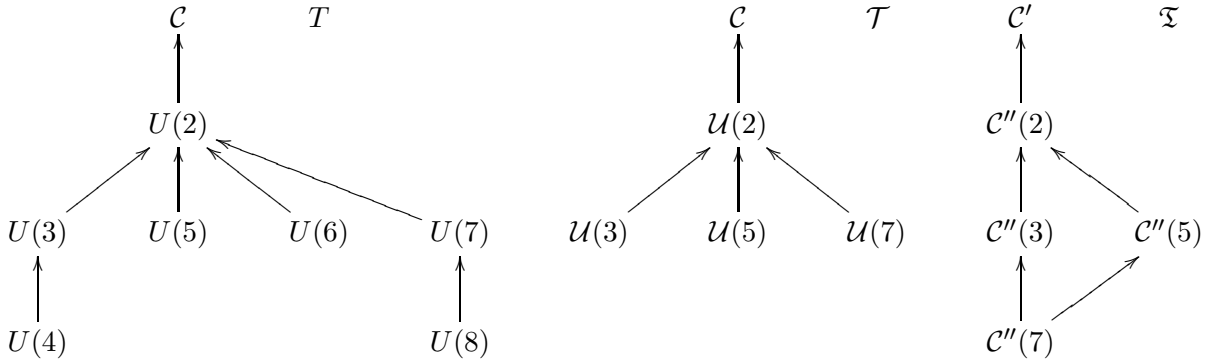

For the sake of simplicity, we perform some natural identifications. We identify from this point on the vertex sets of $T$, $\mc{T}$, and $\mf{T}$ with the sets $\mc{C}'' \cup \{\mc{C}'\}$, $\mc{J} \cup \{\mc{C}'\}$, and $\mc{J} \cup \{\mc{C}'\}$, respectively (recall that $U(\mc{C}') = \mc{C}$ and $\mc{U}(\mc{C}') = \mc{C}$). With these identifications, the arborescence $\mc{T}$ and the acyclic directed graph $\mf{T}$ has the same vertex set, which makes it possible to depict them at once. We have depicted $T$, $\mc{T}$, and $\mf{T}$ considering these identifications in Figure \ref{fig:T_mcT_mfT_identified}.

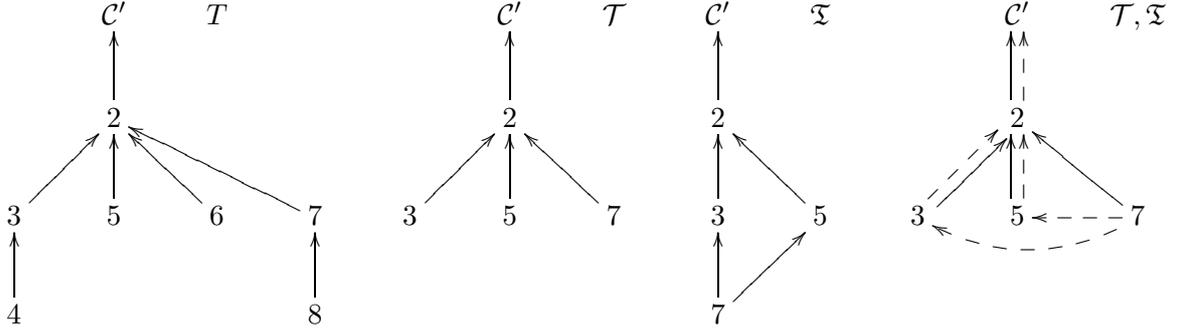
\begin{figure}[!h!t!b]
  \begin{align*}
  \xymatrix{
            & \mc{C}'  & T         &            &           & \mc{C}'  & \mc{T}    & \mc{C}'          & \mf{T}    &                      & \mc{C}'              & \mc{T}, \mf{T}       \\
            & 2 \ar[u] &           &            &           & 2 \ar[u] &           & 2 \ar[u]         &           &                      & 2 \ar@<.5ex>[u]
                                                                                                                                           \ar@<-.5ex>@{-->}[u] &                      \\
  3 \ar[ru] & 5 \ar[u] & 6 \ar[lu] & 7 \ar[llu] & 3 \ar[ru] & 5 \ar[u] & 7 \ar[lu] & 3 \ar[u]         & 5 \ar[lu] & 3 \ar@<-.5ex>[ru]
                                                                                                                    \ar@<.5ex>@{-->}[ru] & 5 \ar@<.5ex>[u]
                                                                                                                                           \ar@<-.5ex>@{-->}[u] & 7 \ar[lu]
                                                                                                                                                                  \ar@{-->}[l]
                                                                                                                                                                  \ar@/^1pc/@{-->}[ll] \\
  4 \ar[u]  &          &           & 8 \ar[u]   &           &          &           & 7 \ar[u] \ar[ru] &           &                      &                      &                      \\
  }
  \end{align*}
  \caption{The arborescences $T$ and $\mc{T}$ associated to the laminar families \eqref{eq:laminar_U_mcU} and the acyclic directed graph $\mf{T}$ for \eqref{eq:xypic_U} considering the straightforward identifications of the vertex sets.}
  \label{fig:T_mcT_mfT_identified}
\end{figure}

The following proposition states that for $j \in \mc{J}$ the condition $W(j) \subseteq \mc{C}''(j)$ can be expressed in terms of $\mc{T}$ and $\mf{T}$. For a directed graph $D = (V, A)$ and $j \in V$ let us denote by $R_D(j)$ the set of vertices from which it is possible to reach $j$ in $D$, i.e.,
\begin{align}\label{eq:def_of_RD(j)}
R_D(j) = \{ \wt{j} \in V ~|~ \mbox{there exists a directed path from $\wt{j}$ to $j$ in $D$}\}.
\end{align}
It is trivial that for all $j \in \mc{J}$ we have $R_{\mc{T}}(j) \subseteq R_{\mf{T}}(j)$.

\begin{proposition}\label{prop:equiv_Wj_subset_C''j}
Assume that $(\mc{C}, \mc{R})$ satisfies \eqref{eq:tree-like_1}. Let $\mc{J}$ be as in \eqref{eq:J} and for $j \in \mc{C}''$ let $W(j)$ be as in \eqref{eq:def_of_W(j)}. Also, let the arborescence $\mc{T}$ and acyclic directed graph $\mf{T}$ be as above. Fix $j \in \mc{J}$. Then $W(j) \subseteq \mc{C}''(j)$ if and only if $R_{\mc{T}}(j) = R_{\mf{T}}(j)$, where $R_{\mc{T}}(j)$ and $R_{\mf{T}}(j)$ are understood in accordance with \eqref{eq:def_of_RD(j)}.
\end{proposition}
\begin{proof}
Assume first that $R_{\mc{T}}(j) = R_{\mf{T}}(j)$ and suppose by contradiction that $W(j) \bs \mc{C}''(j) \neq \emptyset$. Let $i \in W(j) \bs \mc{C}''(j)$, $P_j \in \orarr{i,j}$, and $P_{\mc{C}'} \in \orarr{i,\mc{C}'}$ be such that $V[P_j] \cap V[P_{\mc{C}'}] = \{i\}$ (see Proposition \ref{prop:equiv_to_i_in_W(j)}). Also, let $P \in \orarr{j,\mc{C}'}$ be such that $V[P] \cap \mc{C}''(j) = \{j\}$ (recall that $\varrho^{\outarc}(j) \cap \varrho^{\outarc}(\mc{C}''(j)) \neq \emptyset$). Then both $P_{\mc{C}'}$ and $\con(P_j,P)$ are directed paths from $i$ to $\mc{C}'$ and these two cannot have any common vertex in $\mc{C}''(j)$. Thus, there does not exist $j' \in \mc{C}''(j)$ such that $i \in U(j')$, and consequently $i \notin \mc{U}(j)$. Let $\ol{j} \in \mc{J}$ be such that $i \in \mc{C}''(\ol{j})$. Then $\ol{j} \in R_{\mf{T}}(j)$ and $\ol{j} \notin R_{\mc{T}}(j)$, which contradicts $R_{\mc{T}}(j) = R_{\mf{T}}(j)$.

To show the other direction, assume that $W(j) \subseteq \mc{C}''(j)$ and suppose by contradiction $i \in \mc{C}''$ is such that $\orarr{i,j} \neq \emptyset$ and there does not exist $j' \in \mc{C}''(j)$ such that $i \in U(j')$. Moreover, assume that $i$ is such that $\orarr{p_T(i),j} = \emptyset$, where $p_T(i) \in \mc{C}$ is the \emph{parent} of $i$ in $T$ (i.e., $p_T(i)$ is the second vertex on the unique directed path from $i$ to $\mc{C}'$ in $T$). Since $i \in U(i)$, it follows that $i \notin \mc{C}''(j)$. We will show that $i \in W(j)$, which will thus contradict $W(j) \subseteq \mc{C}''(j)$. To show the inclusion $i \in W(j)$, note that the set $\{ i' \in \mc{C}'' \bs \{i\} ~|~ i \in U(i')\}$ coincides with $V[P] \bs \{i, \mc{C}'\}$, where $P$ is the unique directed path from $i$ to $\mc{C}'$ in $T$. However, by our assumption on $i$, for all $i' \in V[P] \bs \{i, \mc{C}'\}$ we have $\orarr{i',j} = \emptyset$. Thus, taking also into account Proposition \ref{prop:equiv_to_i_in_W(j)} and \eqref{eq:i_in_Q(j)_equiv} in the proof of that proposition, we obtain that $i \in W(j)$. This concludes the proof.
\end{proof}

For the reaction network \eqref{eq:xypic_U} we have
\begin{align*}
R_{\mc{T}}(2) = \{2, 3, 5, 7\} &=          \{2, 3, 5, 7\} = R_{\mf{T}}(2), \\
R_{\mc{T}}(3) = \{3\}          &\subsetneq \{3, 7\}       = R_{\mf{T}}(3), \\
R_{\mc{T}}(5) = \{5\}          &\subsetneq \{5, 7\}       = R_{\mf{T}}(5), ~ \mbox{and} ~ \\
R_{\mc{T}}(7) = \{7\}          &=          \{7\}          = R_{\mf{T}}(7), \\
\end{align*}
which is indeed in accordance with \eqref{eq:W_(n)subseteq_C''}.

We conclude this section by some remarks on the collection $\{U(i) ~|~ i \in \mc{C}''\} \cup \{\mc{C}\}$. First, note that not only $\{U(i) ~|~ i \in \mc{C}''\} \cup \{\mc{C}\}$ determines $T$, but also conversely. Indeed, it is obvious that $U(i) = R_T(i)$ for all $i \in \mc{C}$. In graph theory, if $j \in U(i)$ then we say that $i$ \emph{postdominates} $j$. The arborescence $T$ is called the \emph{postdominator tree}. See e.g. \cite{georgiadis:tarjan:werneck:2006} for more on algorithmic issues concerning the dominator/postdominator tree.

\section*{Acknowledgements}

The author is grateful to Gy\"orgy Michaletzky for the fruitful discussions while facing the difficulties during the attempts of proving the main result.


\appendix

\section{Some basic notions from the theory of directed graphs}\label{sec:app_directed_graphs}

We collect in this section those standard concepts and notations from graph theory that are used throughout this paper. The notions are taken from \cite[Sections 3.2, 9.1, and 10.1]{schrijver:2003}.

Let $D=(V,A)$ be a directed graph without multiple arcs throughout this section.

For $i_0, i_1, \ldots, i_l \in V$ with $l \in \Z_{\geq0}$, we say that $P = (i_0, i_1, \ldots, i_l)$ is a \emph{directed walk from $i_0$ to $i_l$}, or just a \emph{directed walk}, if $(i_k, i_{k+1}) \in A$ for all $k \in \{0, 1, \ldots, l-1\}$. The \emph{length} of a directed walk $P = (i_0, i_1, \ldots, i_l)$, denoted by $\len(P)$, is $l$. For a directed walk $P = (i_0, i_1, \ldots, i_l)$, we denote by $V[P]$ the vertex set $\{i_0, i_1, \ldots, i_l\}$ (we use the notation $V[P]$ even if the vertex set of the directed graph in question is denoted by some other symbol than $V$). If $i \in V[P]$ for a directed walk $P$ then we say that $P$ \emph{traverses} $i$, while if $i \notin V[P]$ then we say that $P$ \emph{avoids} $i$. One can define traversing and avoiding a set $U \subseteq V$ similarly. A directed walk $P = (i_0, i_1, \ldots, i_l)$ is said to be a \emph{directed path} if $i_0, i_1, \ldots, i_l$ are all distinct. For $i, j \in V$, we denote by $\orarr{i,j}$ the set of directed paths from $i$ to $j$. For a directed path $P = (i_0, i_1, \ldots, i_l)$ and $0 \leq k \leq m \leq l$ we denote by $P^{i_k:i_m}$ the directed path $(i_k, i_{k+1}, \ldots, i_m)$. A directed walk $P = (i_0, i_1, \ldots, i_l)$ is said to be a \emph{directed circuit} if $l \geq 1$, $i_0 = i_l$, and $i_0, i_1, \ldots, i_{l-1}$ are all distinct.

For the directed walks $P_1 = (i_0, i_1, \ldots, i_l)$ and $P_2 = (j_0, j_1, \ldots, j_k)$ with $i_l = j_0$, we denote by $\con(P_1,P_2)$ their \emph{concatenation}, which is defined by $\con(P_1,P_2) = (i_0, \ldots, i_l, j_1, \ldots, j_k)$. Clearly, $\con(P_1,P_2)$ is then a directed walk from $i_0$ to $j_k$.

The directed graph $D$ is called \emph{strongly connected} if for all $i, j \in V$ we have $\orarr{i,j} \neq \emptyset$, while it is called \emph{weakly connected} if the underlying undirected graph is connected. The maximal strongly connected subgraphs of $D$ are called the \emph{strong components} of $D$, while the maximal weakly connected subgraphs of $D$ are called the \emph{weak components} of $D$. An \emph{absorbing strong component} of $D$ is a strong component such that there is no arc, which leaves it.

The above definitions are in accordance with the ones in \cite[Section 3.2]{schrijver:2003}, with the only difference that we have defined a directed walk as a sequence of vertices rather than an alternating sequence of vertices and arcs. This choice is made here, because it suffices all our purposes in this paper if we restrict our attention to directed graphs without multiple arcs.

We use a corollary of the following theorem several times in Section \ref{sec:general_forall}. See \cite[Corollary 9.1a]{schrijver:2003} for a proof of Menger's Theorem.

\begin{theorem}[Menger's Theorem \cite{schrijver:2003}]\label{thm:menger}
Let $D = (V,A)$ be a directed graph and let $s$ and $t$ be two nonadjacent vertices
of $D$. Then the maximum number of internally vertex-disjoint $s - t$ paths is
equal to the minimum size of an $s - t$ vertex-cut.
\end{theorem}

For $U \subseteq V$, the set of arcs, which \emph{enter} $U$ and \emph{leave} $U$ are defined by
\begin{align*}
\varrho^{\inarc}_D(U)  &= \{ (i,j) \in A ~|~ i \in V \bs U, j \in U \}~\mbox{and} \\
\varrho^{\outarc}_D(U) &= \{ (i,j) \in A ~|~ i \in U, j \in V \bs U \},
\end{align*}
respectively. Denote by $2^V$ the power set of V. For a function $z : A \to \R$, define the function $\excess_z:2^V\to\R$ by
\begin{eqnarray*}
\excess_z(U) = z(\varrho^{\inarc}(U)) - z(\varrho^{\outarc}(U))~~(U\subseteq V),
\end{eqnarray*}
where $z(\varrho^{\inarc}(U))$ and $z(\varrho^{\outarc}(U))$ are understood in accordance with \eqref{eq:p(X_0)=sump(x)}. For $i\in V$ we use the notations $\varrho^{\inarc}(i),\varrho^{\outarc}(i)$, and $\excess_z(i)$ instead of $\varrho^{\inarc}(\{i\}),\varrho^{\outarc}(\{i\})$, and $\excess_z(\{i\})$, respectively. Note that $\excess_z(V) = \excess_z(\emptyset) = 0$.

An important and frequently used observation is that
\begin{eqnarray}\label{eq:excess_additive}
\excess_z(U) = \sum_{i\in U} \excess_z(i) ~ \mbox{for all} ~ U \subseteq V.
\end{eqnarray}
Thus, the excess function satisfies \eqref{eq:p(X_0)=sump(x)}.

For a function $h : V \to \R$, a function $z : A \to \R$ is called an \emph{$h$-transshipment} if $\excess_z = h$.

Let us define $\mc{A}_{D}(U)$ and $\mc{T}_{D}(U)$ by
\begin{align}
\label{eq:ADU} \mc{A}_{D}(U) &= \left\{ \wt{A} \subseteq A ~\bigg|~ \begin{array}{l} |\varrho^{\outarc}_{(V,\wt{A})}(k)|=0 ~ \mbox{for all} ~ k \in U ~ \mbox{and} \\
                                                                    |\varrho^{\outarc}_{(V,\wt{A})}(k)|=1 ~ \mbox{for all} ~ k \in V \bs U \\ \end{array} \right\} ~ \mbox{and} \\
\label{eq:TDU} \mc{T}_{D}(U) &= \{ \wt{A} \in \mc{A}_D(U) ~|~ (V,\wt{A}) ~ \mbox{is acyclic}\},
\end{align}
respectively. (A directed graph is called \emph{acyclic} if it has no directed circuits.) The elements of $\mc{T}_D(U)$ are called \emph{$U$-branchings in $D$} (or more precisely \emph{$U$-inbranchings in $D$}), the set $U$ being called the \emph{root set}. If $U$ is the singleton $\{j\}$ for some $j \in V$ then a $U$-branching $\wt{A}$ is called a \emph{$j$-arborescence} (or more precisely a \emph{$j$-inarborescence}). Clearly, if $U = \{j_1, \ldots, j_k\}$ for some positive integer $k$ then for a $U$-branching $\wt{A}$, $(V,\wt{A})$ has $k$ weak components, each of them is corresponding to an element of $U$. Denote these weak components by $(V^{j_1},\wt{A}^{j_1}), \ldots,(V^{j_k},\wt{A}^{j_k})$. Then $\wt{A}^j$ is a $j$-arborescence in the directed graph $(V^j,\{(i, i') \in A ~|~ i, i' \in V^j\})$ for all $j \in U$. See Figure \ref{fig:branching} for an illustration of these notions.

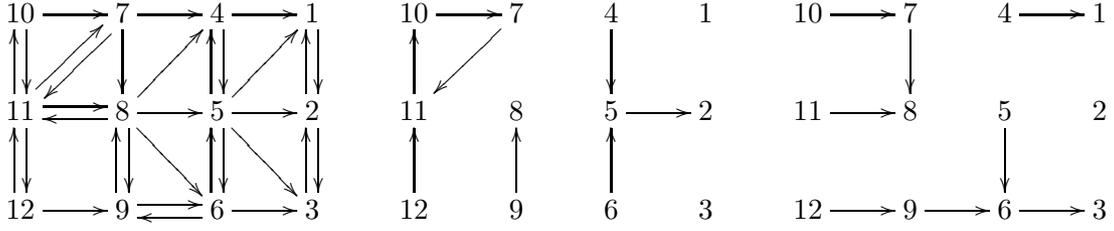
\begin{figure}[!h!t!b]
  \begin{align*}
  \xymatrix{
  10 \ar[r]
     \ar@<.5ex>[d]    &   7 \ar[r]
                            \ar[d]
                            \ar@<.5ex>[ld]   &   4 \ar[r]
                                                   \ar@<.5ex>[d]   &   1 \ar@<.5ex>[d]   & 10 \ar[r] & 7 \ar[dl] & 4 \ar[d] & 1 &   10 \ar[r]    &   7 \ar[d]   &   4 \ar[r]   &   1 \\
  11 \ar@<.5ex>[u]
     \ar@<.5ex>[ur]
     \ar@<.5ex>[r]
     \ar@<.5ex>[d]    &   8 \ar[ru]
                            \ar[r]
                            \ar@<.5ex>[l]
                            \ar@<.5ex>[d]
                            \ar[dr]          &   5 \ar[ur]
                                                   \ar[r]
                                                   \ar[dr]
                                                   \ar@<.5ex>[u]
                                                   \ar@<.5ex>[d]   &   2 \ar@<.5ex>[u]
                                                                         \ar@<.5ex>[d]   & 11 \ar[u] & 8         & 5 \ar[r] & 2 &   11 \ar[r]    &   8         &   5 \ar[d]   &   2 \\
  12 \ar[r]
     \ar@<.5ex>[u]    &   9 \ar@<.5ex>[r]
                            \ar@<.5ex>[u]    &   6 \ar@<.5ex>[l]
                                                   \ar@<.5ex>[u]
                                                   \ar[r]          &   3 \ar@<.5ex>[u]   & 12 \ar[u] & 9 \ar[u]  & 6 \ar[u] & 3 &   12 \ar[r]    &   9 \ar[r]   &   6 \ar[r]   &   3 \\
  }
  \end{align*}
  \caption{A directed graph $D$ with vertex set $\{1, \ldots, 12\}$ (on the left) and two elements of $\mc{A}_D(\{1,2,3,8\})$ (in the middle and on the right). The one in the middle is \emph{not} acyclic, and hence, it does not belong to $\mc{T}_D(\{1,2,3,8\})$, while the one on the right is acyclic, so that one is a $\{1,2,3,8\}$-branching in $D$. It is also apparent that the latter one can be decomposed into four arborescences, the roots of these four arborescences are $1$, $2$, $3$, and $8$, respectively.}
  \label{fig:branching}
\end{figure}

We also use the following notations in Section \ref{sec:general_forall} and Subsection \ref{subsec:matrix-tree-thm}. For $i, j \in V$ let us define $\mc{A}^{ij}_{D}(U)$ and $\mc{T}^{ij}_{D}(U)$ by
\begin{align}
\label{eq:ADUij} \mc{A}^{ij}_{D}(U) &= \{ \wt{A} \in \mc{A}_{D}(U) ~|~ \mbox{there exists a directed path from $i$ to $j$ in $(V,\wt{A})$} \} ~ \mbox{and} \\
\label{eq:TDUij} \mc{T}^{ij}_{D}(U) &= \{ \wt{A} \in \mc{T}_{D}(U) ~|~ \mbox{there exists a directed path from $i$ to $j$ in $(V,\wt{A})$} \},
\end{align}
respectively.

\section{Proof of Theorem \ref{thm:exists_positive_h-transshipment}}\label{sec:app_positive_h-transshipment}

The aim of this section is to prove Theorem \ref{thm:exists_positive_h-transshipment}. The main tool we will use in that proof is the following theorem (see \cite[Corollary 11.2f]{schrijver:2003}), which is a well-known consequence of the Hoffman's Theorem (see \cite[Theorem 11.2]{schrijver:2003}).

\begin{theorem}\label{thm:cor_of_Hoffman}
Let $D = (V,A)$ be a directed graph, let $d,c: A \to \R$ with $d \leq c$ and let $h: V \to \R$ with $h(V) = 0$. Then there exists an $h$-transshipment $z$ with $d \leq z \leq c$ if and only if
\begin{align*}
c(\varrho^{\inarc}(U)) - d(\varrho^{\outarc}(U)) \geq h(U) ~ \mbox{for all} ~ U \subseteq V.
\end{align*}
\end{theorem}

Note that Theorem \ref{thm:exists_positive_h-transshipment} characterises the existence of a \emph{positive} $h$-transshipment for a function $h : V \to \R$ with $h(V) = 0$. Compare this to the characterisation of the existence of a \emph{nonnegative} $h$-transshipment (see \cite[Corollary 11.2h]{schrijver:2003}).

\begin{proofof}[Proof of Theorem \ref{thm:exists_positive_h-transshipment}]
To show that \eqref{eq:h(U)<0} is necessary, let $z: A \to \RPL$ be an $h$-transshipment and $\emptyset \neq U \subsetneq V$ with $\varrho^{\inarc}(U) = \emptyset$. Then
\begin{align*}
h(U) = \excess_z(U) = z(\varrho^{\inarc}(U)) - z(\varrho^{\outarc}(U)) = - z(\varrho^{\outarc}(U)) < 0,
\end{align*}
where the inequality holds, because $\varrho^{\outarc}(U)$ and $\varrho^{\inarc}(U)$ cannot be empty at the same time ($D$ is assumed to be weakly connected and $\emptyset \neq U \subsetneq V$) and the values of $z$ are positive.

To show the sufficiency of \eqref{eq:h(U)<0}, assume for the rest of this proof that \eqref{eq:h(U)<0} holds. Clearly, the existence of a positive $h$-transshipment is equivalent to the existence of $0 < \eps \leq K$ such that there exists an $h$-transshipment $z$ with $\eps \leq z \leq K$. By Theorem \ref{thm:cor_of_Hoffman}, the latter is equivalent to the existence of $0 < \eps \leq K$ such that
\begin{align}\label{eq:k-eps>=h}
K |\varrho^{\inarc}(U)| - \eps |\varrho^{\outarc}(U)| \geq h(U) ~ \mbox{for all} ~ U \subseteq V.
\end{align}

Let
\begin{align}
\label{eq:def_eps} \eps &= \min \left( \left\{- \frac{h(U)}{|\varrho^{\outarc}(U)|} ~\bigg|~ \emptyset \neq U \subsetneq V ~ \mbox{and} ~ \varrho^{\inarc}(U) = \emptyset \right\} \cup \{ 1 \} \right) ~ \mbox{and} \\
\label{eq:def_K} K &= \max \left( \left\{ \frac{h(U) + \eps |\varrho^{\outarc}(U)|}{|\varrho^{\inarc}(U)|} ~\bigg|~ \emptyset \neq U \subsetneq V ~ \mbox{and} ~ \varrho^{\inarc}(U) \neq \emptyset \right\} \cup \{ \eps \} \right).
\end{align}
Note that $\eps > 0$ is guaranteed by \eqref{eq:h(U)<0}. Also, we have $\eps \leq K$. We show that \eqref{eq:k-eps>=h} holds with these specific choices of $\eps$ and $K$.

Both for $U = \emptyset$ and $U = V$ we have $\varrho^{\inarc}(U) = \varrho^{\outarc}(U) = \emptyset$ and $h(U) = 0$, hence \eqref{eq:k-eps>=h} holds in both cases.

Fix for the rest of this proof $\emptyset \neq U \subsetneq V$. In case $\varrho^{\inarc}(U) = \emptyset$, \eqref{eq:k-eps>=h} is a consequence of \eqref{eq:def_eps}, while in case $\varrho^{\inarc}(U) \neq \emptyset$, \eqref{eq:k-eps>=h} follows from \eqref{eq:def_K}.
\end{proofof}

\section{The Matrix-Tree Theorem}\label{sec:app_mtx-tree-thm}

There are several versions of the Matrix-Tree Theorem. The one we present in Subsection \ref{subsec:matrix-tree-thm} appears in \cite[Appendix]{leighton:rivest:1983} as a tool for proving the Markov Chain Tree Theorem. This is slight generalization of Tutte's result (see \cite[Theorem 3.6]{tutte:1948}), while it is a special case of the All Minors Matrix Tree Theorem (see \cite{chaiken}). The previous applications of the Matrix-Tree Theorem in CRNT used Tutte's version (see \cite{craciun:dickenstein:shiu:sturmfels:2008, dickenstein:perezmillan:2011, feliu:wiuf:2012, feliu:wiuf:2013, karp:perezmillan:dasgupta:dickenstein:gunawardena:2012, thomson:gunawardena:2009}), while we need a slightly more general variation in Section \ref{sec:general_forall}. We also provide a direct and elementary proof of the Matrix-Tree Theorem in Subsection \ref{subsec:matrix-tree-thm}, which was worked out by György Michaletzky and the author of this paper.

\subsection{A lemma on the number of inversions in bijections}\label{subsec:app_mtxtree_1}

In the proof of the Matrix-Tree Theorem, we use a purely algebraic lemma, which is a special case of a result in \cite{chaiken}. For the sake of completeness, we also present the proof of this lemma.

Fix a positive integer $n$ for this subsection. Let $W_1$ and $W_2$ be nonempty subsets of $V = \{1, \ldots, n\}$ such that $|W_1| = |W_2|$. For a bijection $\pi : W_1 \to W_2$, we say that $k \in W_1$ and $k' \in W_1$ are in \emph{inversion} if $k < k'$ and $\pi(k) > \pi(k')$. Denote by $\nu(\pi)$ the \emph{number of inversions} in $\pi$, i.e.,
\begin{align*}
\nu(\pi) = |\{(k,k') \in W_1 \times W_1 ~|~ k < k' ~ \mbox{and} ~ \pi(k) > \pi(k')\}|.
\end{align*}
As usual, define the \emph{sign} $\sgn(\pi)$ of the bijection $\pi$ by $\sgn(\pi) = (-1)^{\nu(\pi)}$.

\begin{lemma}\label{lemma:sgnsigma}
Let $i,j \in V$. Let $\sigma : V \bs \{j\} \to V \bs \{i\}$ be a bijection and define the permutation $\ol{\sigma} : V \to V$ by
\begin{align*}
\ol{\sigma}(k) =
\begin{cases}
\sigma(k), & \mbox{if}~k \in V \bs \{j\}, \\
i,         & \mbox{if}~k = j                      \\
\end{cases}
(k \in V).
\end{align*}
Then $\sgn(\ol{\sigma}) = (-1)^{i+j} \sgn(\sigma)$.
\end{lemma}
\begin{proof}
We prove this lemma by induction on $j$. To prove the initial step of the induction, let $j=1$. Clearly, $\ol{\sigma}$ inherits all the inversions of $\sigma$. Also, since $\ol{\sigma}(1) = i$, there are exactly $i-1$ elements in $V \bs \{1\}$ that are in inversion with $1$ under $\ol{\sigma}$. Hence, $\nu(\ol{\sigma}) = \nu(\sigma) + (i-1)$. As a consequence,
\begin{align*}
\sgn(\ol{\sigma}) = (-1)^{\nu(\ol{\sigma})} = (-1)^{\nu(\sigma) + (i-1)} = (-1)^{i-1} (-1)^{\nu(\sigma)} = (-1)^{i+1} \sgn(\sigma).
\end{align*}

To prove the inductive step, fix $2 \leq j \leq n$ and assume that the lemma holds with $j-1$ instead of $j$. Let us define $\pi : V \bs \{j-1\} \to V \bs \{i\}$ by
\begin{align*}
\pi(k) =
\begin{cases}
\sigma(k),   & \mbox{if} ~ k \in V \bs \{j-1,j\}, \\
\sigma(j-1), & \mbox{if} ~ k = j                  \\
\end{cases}
(k \in V \bs \{j-1\}).
\end{align*}
Clearly, we have $\nu(\pi) = \nu(\sigma)$. Also, let us define $\ol{\pi} : V \to V$ by
\begin{align*}
\ol{\pi}(k) =
\begin{cases}
\pi(k), & \mbox{if} ~ k \in V \bs \{j-1\}, \\
i,      & \mbox{if} ~ k = j-1              \\
\end{cases}
(k \in V).
\end{align*}
Note that we have defined $\ol{\pi}$ in such a way that $\ol{\pi}|_{V \bs \{j-1, j\}} = \ol{\sigma}|_{V \bs \{j-1, j\}}$, $\ol{\pi}(j) = \ol{\sigma}(j-1)$, and $\ol{\pi}(j-1) = \ol{\sigma}(j) = i$. Thus, for $k, k' \in V$ such that $k < k'$ we have
\begin{align*}
\ol{\pi}(k) > \ol{\pi}(k') ~ \mbox{if and only if} ~
\begin{cases}
\ol{\sigma}(k)   > \ol{\sigma}(k'),  & \mbox{if} ~ k, k' \in V \bs \{j-1, j\},                                  \\
\ol{\sigma}(k)   > \ol{\sigma}(j-1), & \mbox{if} ~ k \in \{1, \ldots, j-2\}, k' = j,                            \\
\ol{\sigma}(k)   > \ol{\sigma}(j),   & \mbox{if} ~ k \in \{1, \ldots, j-2\}, k' = j-1,                          \\
\ol{\sigma}(j-1) > \ol{\sigma}(k'),  & \mbox{if} ~ k = j,                    k' \in \{j+1, \ldots, n\},         \\
\ol{\sigma}(j)   > \ol{\sigma}(k'),  & \mbox{if} ~ k = j-1,                  k' \{j+1, \ldots, n\}, ~\mbox{and} \\
\ol{\sigma}(j-1) < \ol{\sigma}(j),   & \mbox{if} ~ k =j-1,                   k' = j.                            \\
\end{cases}
\end{align*}
Hence,
\begin{align}\label{eq:NuPi=NuSigma+-1}
\nu(\ol{\pi}) =
\begin{cases}
\nu(\ol{\sigma}) + 1, & \mbox{if} ~ \ol{\sigma}(j-1) < \ol{\sigma}(j), \\
\nu(\ol{\sigma}) - 1, & \mbox{if} ~ \ol{\sigma}(j-1) > \ol{\sigma}(j). \\
\end{cases}
\end{align}
Therefore,
\begin{align*}
\sgn(\ol{\sigma}) = (-1)^{\nu(\ol{\sigma})} \stackrel{\eqref{eq:NuPi=NuSigma+-1}}{=} -(-1)^{\nu(\ol{\pi})} = -(-1)^{i + (j-1)} \sgn(\pi) = (-1)^{i+j} \sgn(\sigma),
\end{align*}
where we used the inductive hypothesis for $\ol{\pi}$. This concludes the proof.
\end{proof}

The following corollary is a direct consequence of Lemma \ref{lemma:sgnsigma}.

\begin{corollary}\label{cor:sgnsigma}
Let $Q \subseteq V$ and $i,j \in V \bs Q$. Let $\sigma : V \bs (Q \cup \{j\}) \to V \bs (Q \cup \{i\})$ be a bijection and define the permutation $\ol{\sigma} : V \bs Q \to V \bs Q$ by
\begin{align*}
\ol{\sigma}(k) =
\begin{cases}
\sigma(k), & \mbox{if}~k \in V \bs (Q \cup \{j\}), \\
i,         & \mbox{if}~k = j                      \\
\end{cases}
(k \in V \bs Q).
\end{align*}
Then $\sgn(\ol{\sigma}) = (-1)^{i+j} \sgn(\sigma)$.
\end{corollary}

\subsection{The Matrix-Tree Theorem}\label{subsec:matrix-tree-thm}

Fix a positive integer $n$ for this subsection. Associate to a matrix $Z=(z_{ij})_{i,j=1}^n \in \Rnn$ the directed graph $D(Z) = (V(Z), A(Z))$, where $V(Z)=\{1,\ldots,n\}$ and $A(Z)=\{(i,j)\in V(Z) \times V(Z)~|~z_{ij}\neq0\}$. Also, for $Q_1, Q_2 \subseteq \{1,\ldots,n\}$ with $|Q_1| = |Q_2|$, denote by $d_{Q_1, Q_2}(Z)$ the determinant of that matrix, which is obtained from $Z$ by deleting the rows with index in $Q_1$ and the columns with index in $Q_2$.

\begin{theorem}[Matrix-Tree Theorem]\label{thm:matrixtree}
Let $Z=(z_{ij})_{i,j=1}^n\in\Rnn$ be a matrix that satisfy
\begin{align}\label{eq:rowsums=0}
\sum_{j=1}^n z_{ij} = 0 ~ \mbox{for all} ~ i \in \{1, \ldots, n\}.
\end{align}
Fix $Q \subseteq \{1,\ldots,n\}$ and $i, j \in \{1,\ldots,n\} \bs Q$. Then
\begin{align*}
d_{Q \cup \{j\}, Q \cup \{i\}}(Z)=(-1)^{i+j} (-1)^{n-|Q|-1} \sum_{\wt{A} \in \mc{T}^{ij}_{D(Z)}(Q \cup \{j\})} z_{\wt{A}},
\end{align*}
where $\mc{T}^{ij}_{D(Z)}(Q \cup \{j\})$ is understood as in \eqref{eq:TDUij} and the symbol $z_{\wt{A}}$ is a shorthand notation for the product $\prod_{a \in \wt{A}} z_a$.
\end{theorem}
\begin{proof}
For shorthand notation, let us denote by $S_{ij}$ the set of bijections from $V \bs (Q \cup \{j\})$ to $V \bs (Q \cup \{i\})$. For an element $\sigma$ of $S_{ij}$, denote by $\ol{\sigma}$ the permutation of $V \bs Q$, which is defined by
\begin{align*}
\ol{\sigma}(k) =
\begin{cases}
\sigma(k), & \mbox{if}~k \in V \bs (Q \cup \{j\}), \\
i,         & \mbox{if}~k = j                      \\
\end{cases}
(k \in V \bs Q).
\end{align*}
Since $\ol{\sigma}$ is a permutation, we may consider its decomposition into disjoint cyclic permutations. This also yields a decomposition of $\sigma$ into cyclic permutations and a \lq\lq path bijection'' from $i$ to $j$, where the \lq\lq path bijection'' is coming from the \lq\lq deletion'' of the assignment $j \mapsto i$ from $\ol{\sigma}$. Denote by
\begin{align*}
&h^{\sigma}~\mbox{the \lq\lq path bijection'' component of $\sigma$,}                               \\
&p_{\sigma}~\mbox{the number of cycles of $\sigma$ of length at least $2$,}                         \\
&f^{\sigma}_1,\ldots, f^{\sigma}_{p_{\sigma}}~\mbox{the cycles of $\sigma$ of length at least $2$,} \\
&q_{\sigma}~\mbox{the number of cycles of $\sigma$ of length $1$, and}                              \\
&g^{\sigma}_1,\ldots, g^{\sigma}_{q_{\sigma}}~\mbox{the cycles of $\sigma$ of length $1$.}          \\
\end{align*}
See \eqref{eq:xypic_sigmadecomp} for an illustration of this decomposition. For this specific example, we have $p_{\sigma}=2$ and $q_{\sigma}=3$.
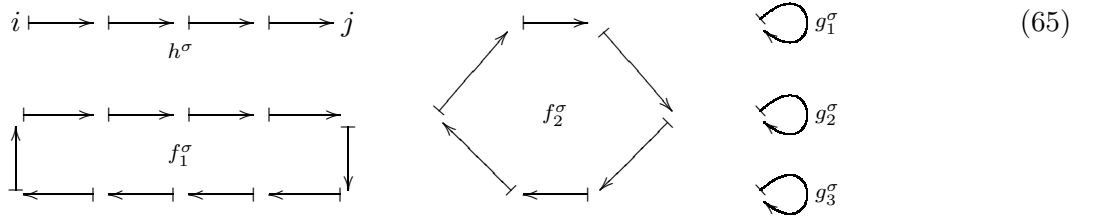
\begin{align}\label{eq:xypic_sigmadecomp}
\xymatrix{
i \ar@{|->}[r] & \ar@{|->}[r]              & \ar@{|->}[r] & \ar@{|->}[r] & j            &                            & \ar@{|->}[r]  & \ar@{|->}[rd] &               & \ar@{|->}@(ur,dr)^{g^{\sigma}_1} \\
  \ar@{|->}[r] & \ar@{|->}[r] & \ar@{|->}[r]
                                \ar@{}|(.7){h^{\sigma}}[u]
                                \ar@{}|{f^{\sigma}_1}[d] & \ar@{|->}[r] & \ar@{|->}[d] & \ar@{|->}[ru]
                                                                           \ar@{}|{f^{\sigma}_2}[rrr] &               &               & \ar@{|->}[ld] & \ar@{|->}@(ur,dr)^{g^{\sigma}_2} \\
  \ar@{|->}[u] & \ar@{|->}[l]              & \ar@{|->}[l]  & \ar@{|->}[l] & \ar@{|->}[l] &                            & \ar@{|->}[lu] & \ar@{|->}[l]  &               & \ar@{|->}@(ur,dr)^{g^{\sigma}_3} \\
}
\end{align}
For notational convenience, we denote by $z_{h^{\sigma}}$ and $z_{f_l^{\sigma}}$ the product of those entries of $Z$ that correspond to the arcs of the directed path $h^{\sigma}$ and the directed circuit $f_l^{\sigma}$, respectively. (Thus, we implicitly identify the bijections $h^{\sigma}$ and $f_l^{\sigma}$ with the corresponding directed path and directed circuit, respectively.) Also, we identify $g_l^{\sigma}$ with the respective vertex, and thus $z_{g_l^{\sigma},g_l^{\sigma}}$ is the corresponding diagonal entry of $Z$. Then
\begin{align*}
d_{Q \cup \{j\}, Q \cup \{i\}}(Z)
      &= \sum_{\sigma \in S_{ij}} \left[ \sgn(\sigma) \cdot \left( \prod_{k \in V \bs (Q \cup \{j\})} z_{k,\sigma(k)} \right) \right] = \\
      &= \sum_{\sigma \in S_{ij}} \left[ \sgn(\sigma) \cdot z_{h^{\sigma}} \cdot \left( \prod_{l=1}^{p_{\sigma}} z_{f^{\sigma}_l} \right) \cdot
      \left( \prod_{l=1}^{q_{\sigma}} z_{g_l^{\sigma},g_l^{\sigma}} \right) \right] \stackrel{\eqref{eq:rowsums=0}}{=} \\
      &\stackrel{\eqref{eq:rowsums=0}}{=} \sum_{\sigma\in S_{ij}} \left[\sgn(\sigma) \cdot z_{h^{\sigma}} \cdot \left(\prod_{l=1}^{p_{\sigma}}z_{f^{\sigma}_l} \right) \cdot
      \left( \prod_{l=1}^{q_{\sigma}}\left\{-\sum_{k\in V \bs \{g^{\sigma}_l\}} z_{g^{\sigma}_l,k} \right\} \right) \right]= \\
      &= \sum_{\sigma\in S_{ij}} \left[ \sgn(\sigma) \cdot(-1)^{q_{\sigma}} \cdot \left( \sum_{\wt{A} \in \mc{A}_{D(Z)}^{ij,\sigma}(Q \cup \{j\})} z_{\wt{A}} \right) \right]= \\
      &= \sum_{\wt{A} \in \mc{A}_{D(Z)}^{ij}(Q \cup \{j\})} z_{\wt{A}} \left[\sum_{\substack{\sigma\in S_{ij} \\
                h^{\sigma} = h^{\wt{A}} \\
                f^{\sigma}_1,\ldots, f^{\sigma}_{p_{\sigma}} \subseteq \wt{A} \\}}
                \sgn(\sigma) \cdot (-1)^{q_{\sigma}} \right], \\
\end{align*}
where $h^{\wt{A}}$ is the unique directed path from $i$ to $j$ in $(V,\wt{A})$ and
\begin{align*}
\mc{A}_{D(Z)}^{ij,\sigma}(Q \cup \{j\}) = \left\{ \wt{A} \in \mc{A}_{D(Z)}^{ij}(Q \cup \{j\}) ~\bigg|~
\begin{array}{l}
f^{\sigma}_1, \ldots, f^{\sigma}_{p_{\sigma}} \subseteq \wt{A} ~ \mbox{and the unique directed} \\
\mbox{path from $i$ to $j$ in $(V,\wt{A})$ is given by $h^{\sigma}$} \\
\end{array}\right\}
\end{align*}
(recall \eqref{eq:ADUij}). Fix $\wt{A} \in \mc{A}_{D(Z)}^{ij}(Q \cup \{j\})$ for the rest of this proof. We claim that
\begin{align}\label{eq:sum=0or1}
\sum_{\substack{\sigma\in S_{ij} \\
                h^{\sigma} = h^{\wt{A}} \\
                f^{\sigma}_1,\ldots, f^{\sigma}_{p_{\sigma}} \subseteq \wt{A} \\}}
                \sgn(\sigma) \cdot (-1)^{q_{\sigma}} = \left\{
                \begin{array}{ll}
                   0,                        & \mbox{if}~\wt{A} \notin \mc{T}_{D(Z)}^{ij}(Q \cup \{j\}), \\
                   (-1)^{i+j}(-1)^{n-|Q|-1}, & \mbox{if}~\wt{A} \in \mc{T}_{D(Z)}^{ij}(Q \cup \{j\}). \\
                \end{array}\right.
\end{align}

By Corollary \ref{cor:sgnsigma},
\begin{align}\label{eq:sgn_sigma*(-1)^q}
\sgn(\sigma) = (-1)^{i+j} \sgn(\ol{\sigma})= (-1)^{i+j} (-1)^{\len(h^{\sigma})} (-1)^{\sum_{l=1}^{p_{\sigma}} (\len(f_l^{\sigma}) - 1)} = (-1)^{i+j} (-1)^{n-|Q|-1-p_{\sigma}-q_{\sigma}}.
\end{align}
Thus, $\sgn(\sigma) \cdot (-1)^{q_{\sigma}} = (-1)^{i+j} (-1)^{n-|Q|-1-p_{\sigma}}$, which depends on $\sigma$ only through $p_{\sigma}$.

If $\wt{A} \in \mc{T}_{D(Z)}^{ij}(Q \cup \{j\})$ (i.e., $\wt{A}$ is acyclic) then the sum on the left hand side of \eqref{eq:sum=0or1} contains only one term (there is only one element $\sigma \in S_{ij}$ for which $h^{\sigma} = h^{\wt{A}}$ and $p_{\sigma} = 0$). Thus, by \eqref{eq:sgn_sigma*(-1)^q}, we obtain \eqref{eq:sum=0or1} for the case $\wt{A} \in \mc{T}_{D(Z)}^{ij}(Q \cup \{j\})$.

Assume for the rest of this proof that $\wt{A} \notin \mc{T}_{D(Z)}^{ij}(Q \cup \{j\})$. Denote by $m$ the number of directed circuits in $\wt{A}$. Since $\wt{A} \notin \mc{T}_{D(Z)}^{ij}(Q \cup \{j\})$, we have $m \geq 1$. With this and \eqref{eq:sgn_sigma*(-1)^q}, we have
\begin{align*}
\sum_{\substack{\sigma\in S_{ij} \\
                h^{\sigma} = h^{\wt{A}} \\
                f^{\sigma}_1,\ldots, f^{\sigma}_{p_{\sigma}}\subseteq\wt{A} \\}}\sgn(\sigma)\cdot(-1)^{q_{\sigma}}&=
(-1)^{i+j}(-1)^{n-|Q|-1}\cdot\sum_{\substack{\sigma\in S_{ij} \\
                h^{\sigma} = h^{\wt{A}} \\
                f^{\sigma}_1,\ldots, f^{\sigma}_{p_{\sigma}}\subseteq\wt{A} \\}}(-1)^{p_{\sigma}}= \\
                &= (-1)^{i+j}(-1)^{n-|Q|-1}\cdot\sum_{k=0}^{m}\binom{m}{k}(-1)^k=0,
\end{align*}
where the last equality follows from the Binomial Theorem.
\end{proof}

\bibliographystyle{plain}
\bibliography{biblio}

\end{document}